\documentclass[preprint,1p]{elsarticle}
\usepackage{algorithm}
\usepackage{algpseudocode}
\usepackage{algorithmicx}
\usepackage{natbib}
\usepackage{amsfonts,latexsym,amsmath, amssymb,amsthm}
\usepackage{mathrsfs}
\usepackage{bm}
\usepackage{color}
\usepackage{chngcntr}
\usepackage{apptools}
\usepackage{clipboard}
\newclipboard{myclipboard_main}
\usepackage[utf8]{inputenc}
\usepackage{graphicx,enumerate}
\usepackage{float} 
\usepackage{subfigure}
\usepackage{enumitem}
\usepackage{cleveref}
\makeatletter
\newenvironment{breakablealgorithm}
{
	\begin{center}
		\refstepcounter{algorithm}
		\hrule height.8pt depth0pt \kern2pt
		\renewcommand{\caption}[2][\relax]{
			{\raggedright\textbf{\ALG@name~\thealgorithm} ##2\par}%
			\ifx\relax##1\relax 
			\addcontentsline{loa}{algorithm}{\protect\numberline{\thealgorithm}##2}%
			\else 
			\addcontentsline{loa}{algorithm}{\protect\numberline{\thealgorithm}##1}%
			\fi
			\kern2pt\hrule\kern2pt
		}
	}{
		\kern2pt\hrule\relax
	\end{center}
}
\makeatother

\newcommand{\e}{\varepsilon}

\def\beaa{\begin{eqnarray*}}
\def\eeaa{\end{eqnarray*}}
\def\ba{\begin{array}}
\def\ea{\end{array}}
\def\be#1{\begin{equation} \label{#1}}
\def \eeq{\end{equation}}

\def\be{{\beta}}

\def\pr{{\partial}}

\def\s{{\sigma}}

\def\R{{\mathbb{R}}}

\def\sgH2{\sigma_H^2}
\def\sgL2{\sigma_L^2}

\newtheorem{theorem}{Theorem}[section]
\newtheorem{lemma}[theorem]{Lemma}
\newtheorem{proposition}[theorem]{Proposition}
\newtheorem{cor}[theorem]{Corollary}

\newtheorem{remark}[theorem]{Remark}

\newtheorem{assumption}[theorem]{Assumption}
\newproof{pf}{proof}
\newcommand\ben{\begin{equation}}
\newcommand\een{\end{equation}}
\newcommand\bea{\begin{eqnarray}}
\newcommand\eea{\end{eqnarray}}

\numberwithin{equation}{section}

\newcommand{\dyc}[1]{{\color{blue}#1}}

\numberwithin{equation}{section}

\allowdisplaybreaks
\begin{document}
\begin{frontmatter}
\title{Double free boundary problem for defaultable corporate bond with credit rating migration risks and their asymptotic behaviors}
\author[1]{Yuchao Dong\fnref{fn1}}
\ead{ycdong@tongji.edu.cn}
\author[1]{Jin Liang\corref{cor1}\fnref{fn2}}
\ead{liang\_jing@tongji.edu.cn}
\author[1,2]{Claude-Michel Brauner\fnref{fn2}}
\ead{claude-michel.brauner@u-bordeaux.fr}
\cortext[cor1]{Corresponding author}
\affiliation[1]{organization={School of Mathematical Sciences, Tongji University},
	city={Shanghai 200092},
	country={China}}
\affiliation[2]{organization={Institut de Math\unexpanded{\'e}matiques de Bordeaux, Universit\unexpanded{\'e} de Bordeaux},
	city={33405 Talence},
	country={France}}
\fntext[fn1]{Yuchao Dong acknowledges partial support from  National Natural Science Foundation of China (No. 12071333 \& No. 12101458)}
\fntext[fn2]{Jin Liang and Claude-Michel Brauner acknowledge partial support from  National Natural Science Foundation of China (No. 12071349)}
\begin{keyword}
	Traveling wave; Free boundary problem; PDE with discontinuous leading order coefficient; Asymptotic behavior; Credit rating migration risk model
\end{keyword}

\begin{abstract} 
In this work, a pricing model for a defaultable corporate bond with credit rating migration risk is established. The model turns out to be a free boundary problem with two free boundaries. The latter are the level sets of the solution but of different kinds. One is from the discontinuous second order term, the other from the obstacle. Existence, uniqueness, and regularity of the solution are obtained.  We also prove that  two free boundaries are  $C^\infty$. The asymptotic behavior of the solution is also considered: we show that it converges to a traveling wave solution when time goes to infinity. Moreover, numerical results are presented. 
\end{abstract}
\end{frontmatter}

\section{Introduction}
 
 Due to the globalization and complexity of financial markets, the credit risks become more and more important and an unstable factor  impacting the market, which might cause a crucial crisis. For example, in the 2008 financial tsunami and the
 2010 European debt crisis, credit rating migration risk played a key role. The first step to managing the risks is modeling and measuring them. Thus, it has attracted
  more and more attention both in academics and in industry to understand these risks, especially default risk and credit rating migration risk.

  Most credit risk research falls into two kinds of framework, namely structure model and intensity one. The intensity model assumes that the risk is due to some  exogenous factors, which are usually modeled by  Markov chains, see \cite{Ma}. In this way, the default and/or migration times are determined by an exogenous transition intensity; see Jarrow, Lando, and Turnbull \cite {JT,JLT}, Duffie and Singleton \cite{DS}, to mention a few. In the  implementation, intensity transition matrices  are usually  obtained from historical statistical data. However, it is well-known that  companies' current financial status plays a crucial role in default and credit rating migrations. For example, the main reason which caused the 2010 European debt crisis was that the
   sovereign debts of several European countries reached an unsustainable level due to their poor economical situation. The crisis happened in these countries because of the downgrading of their credit ratings and the subsequent chain reactions. Therefore, Markov chain model alone cannot fully capture the credit risks.

   To include the endogenous factor, the structural model comes into consideration for credit risk modeling, which could be traced back to Merton \cite{M} in 1974. In such a kind of  models, the reason for credit rating migration and default is  related to the firm's asset value and its obligation. For example,  in Merton's model, it is assumed that the company's  asset value follows a geometric Brownian motion and a default would happen if the asset value drops below the debt at maturity. Thus, the corporate bond, representing the company's obligation, is a contingent claim of the asset value.
     Later, Black and Cox \cite{BC} extended Merton's  model to the so-called first passage-time model, where the
       default would happen whenever the asset value reached a given boundary; see also \cite{Le,LS,LT,BV,TF} for related works.
       Dai et al.  \cite{DHK} considered an optimal control problem in the case where a bank's asset is opaque.

            Using the structural model, Liang and Zeng \cite{LZ}  studied the pricing problem of the corporate bond with credit rating migration risk, where a predetermined migration threshold is given to
            divide asset value into high and low rating regions, in which the asset value follows different stochastic processes. Hu, Liang, and Wu \cite{HLW} further developed this model, where the migration boundary is a free boundary governed by a ratio of the firm's asset value and debt. Some theoretical results and traveling wave properties are also obtained in \cite{LWH16}. Li, Zhang, and Hu \cite{li2018convergence} studied the numerical method for solving related variational inequality. Later, Fu, Chen, and Liang \cite{FCL} provided more  mathematical analysis and detailed description of the free migration boundary.  More extension of this model is considered in \cite{LYCW, YLW, WL, WLH}. Recently, Chen and Liang \cite{CL} also considered the case where upgrade and downgrade boundaries are different. The readers can also refer to the survey paper \cite{CHLY} for a summary.  

However, the reason behind the credit rating migration is the default possibility; hence, it is natural to consider a model involving both the credit rating migration and default risks. In \cite{WLH}, as the first step,  a predetermined default boundary of asset level is considered. In this paper, we will let the default boundary also depend on the ratio between the stock price and bond value. Therefore, the model will contain two free boundaries. Both of these boundaries are the level sets of the solution but of different types. One is from discontinuous leading second order term as in previous credit rating migration works (for example, see \cite{LWH16}); the other is from a more traditional free boundary problem, i.e. obstacle problem.  Using PDE techniques, existence, uniqueness, regularity, and asymptotic behavior of the solution are obtained, which from a theoretical perspective insure the rationality of the model. Numerical results support our theoretical approach. The stability of traveling wave equation will be studied in our future work \cite{BDLL22} using the techniques of \cite{ABLZ20,BHL00,BLZ20}.

 This paper is organized as follows. In Section 2, the model is established and the pricing problem is reduced to a system of two parabolic PDEs with two free boundaries. In Section 3, for the sake of both uniform estimates and asymptotic behavior, we consider a traveling wave solution to the original problem. In Section 4, we use a penalization method and simultaneously a regularization of the coefficient of the 2nd order term to approximate the free boundary problem by a smooth Cauchy problem depending on a small parameter $\e>0$. A series of lemmas are proved in order to establish estimates which are independent of $\e$. The key point is that  the two approximating free boundaries can be separated by a positive distance independent of $\e$. In Section 5, the main results are stated, including the existence, uniqueness, and regularity of the solution. In particular, we prove that  two free boundaries are  $C^\infty$. The asymptotic behavior of the solution as time tends to infinity is examined in Section 6. Finally, a numerical method and some computational results are presented in Section 7.

\section{The Model}\label{Model}
\subsection{Assumptions}
 
Let $(\Omega,\mathcal{F},P)$ be a complete probability space. We
assume that the firm issues a corporate bond, which is a
contingent claim of its value. The stock price of the firm admits different dynamics for different credit ratings.
\begin{assumption}[the firm asset with credit rating migration]
	Let $S_t$ denote the firm's value in the risk neutral world. It
	satisfies
	\begin{equation*}
		dS_t=\left\{\begin{array}{l}rS_tdt+\sigma_HS_tdW_t,\quad \mbox{ in high rating region,}\\
			rS_tdt+\sigma_LS_tdW_t,\quad \mbox{ in low
				rating region,}\end{array}\right.
	\end{equation*}
	where $r$ is the risk free interest rate, which is positive constant, and
\begin{equation}\label{sigmas}
\sigma_{H}<\sigma_{L}
\end{equation}
represent volatilities
	of the firm under the high and low credit grades respectively.
	They are also assumed to be positive constants. $W_t$ is the Brownian
	motion which generates the filtration $\{{\mathcal F}_t\}$.
	\end{assumption}
It is reasonable to assume \eqref{sigmas}, namely that the
volatility in high rating region is lower than the one in the low
rating region. 	The firm issues only one zero coupon corporate bond with face
value $F$. Let $\Phi_t$ denote the discount value of the bond at time
$t$. Therefore, at the
maturity time $T$, an investor can get $\Phi_T=\min\{S_T,F\}$. For simplicity, we assume in the following sections $F=1$. The rating criterion is based on the ratio between the stock price and liability.
\begin{assumption}[the credit rating migration time]
	High and low rating regions are determined by the  proportion between the debt and asset value. The credit rating migration time $\tau_1 $ and $\tau_2 $ are the
	first moments when the firm's grade is downgraded and upgraded
	respectively as follows:
	\begin{eqnarray*}
		&&\tau_1=\inf\{t>0|\Phi_0/S_0<\gamma e^{-\delta T}, \Phi_t/S_t\geqslant \gamma e^{-\delta (T-t)}\},\\
		&& \tau_2=\inf\{t>0|\Phi_0/S_0>\gamma e^{-\delta T}, \Phi_t/S_t\leqslant \gamma e^{-\delta (T-t)}\},
	\end{eqnarray*}
	where $\Phi_t=\Phi_t(S_t,t)$ is a contingent claim with respect to
	$S_t$ and 
	\begin{equation}\label{gamma}
		0<\gamma<1
	\end{equation} 
	is a positive
	constant representing  the threshold proportion of the debt and
	value of the firm's rating. Also
	$$ \delta >0,$$
	is the so-called credit discount rate.
{\color{black}
 In this paper, we also make the assumption that
 \begin{equation}\label{main_sigma}
 	\frac{1}{2}\sigma^2_H <\delta< \frac{1}{2}\sigma^2_L.
 \end{equation}
}
\end{assumption}
Further, we assume that the bond will default when the stock price is too low, compared with the debt.
\begin{assumption}[the defaultable corporate bond]
 The default time is also determined by the proportion of the debt and asset value. Here, we assume that the default happens whenever
 $$
 S_te^{-\delta(T-t)} \leqslant \Phi_t.
 $$
	The default time is defined as 
	$$\tau =\inf\{t>0|\Phi_0>e^{-\delta T}S_0, \Phi_t\geqslant e^{-\delta(T-t)}S_t\}.$$
At the default time, the contract is closed and the investor obtains the cash $e^{-\delta(T-t)}S_t$.
\end{assumption}
\begin{remark}
Condition \eqref{main_sigma} is also assumed in \cite{LWH16} to ensure the existence of the travelling wave equation. In finance, if $\delta$ is too small or too large, it is possible that the company will always be low rating or high rating. To see this, assume that the stock price is 
$$
S_t=e^{rt-\frac12\int_0^t \sigma^2(u)du+\int_0^t \sigma(u)dW_u},
$$
where $\sigma(s)$ is the volatility of the stock taking values in $\{\sigma_H,\sigma_L\}$ depending on whether the stock is low rating or high rating. The present value of the bond is $e^{-r(T-t)}$. Then, the company's  discounted debt-to-asset ratio is 
$$
e^{-\delta t}\frac{e^{-r(T-t)}}{S_t}=e^{-rT}e^{\int_0^t (\frac12\sigma^2(u)-\delta) du- \int_0^t \sigma(u)dW_u}.
$$
If $\delta<\frac{1}{2}\sigma_H^2$, the right hand side will go to $\infty$ as $t \rightarrow \infty$ with probability $1$. This implies that the company will always be low rating in the end.  On the other hand, if $\delta>\frac12\sigma^2_L$, the right hand side will go to $0$ and, hence, the company will always be high rating.

\end{remark}
\subsection{The Cash Flow}  If the bond does not default, once the credit rating migrates before
the maturity $T$, a virtual substitute termination happens, i.e.,
the bond is virtually terminated and substituted by a new one with
a new credit rating. There is a virtual cash flow of the bond.
We denote by {\color{black}  $\Phi_H(S,t)$ and $\Phi_L(S,t)$} the values of the bond
in high and low grades respectively, which are functions of
{\color{black} $S$} and $t$. Then, they are conditional expectations of the following
\begin{align}\label{highholding}
	\nonumber \Phi_H(S,t)=&E\Big[e^{-r(T-t)} \min(S_{T},F)\cdot {\bf 1}_{\{T<\tau_1\wedge\tau\}}\\
	\nonumber&+S_te^{-\delta(T-\tau)}e^{-r(\tau-t)}{\bf 1}_{\{\tau< T\wedge\tau_1\}}\\
	&+e^{-r(\tau_1-t)}\Phi_L(S_{\tau_1} ,\tau_1)\cdot {\bf 1}_{\{\tau_1< T\wedge\tau\}}\Big|S_t=S>\frac 1{\gamma e^{-\delta (T-t)}}\Phi_H(S,t)\Big], \\
	\label{lowholding}
	\nonumber\Phi_L(S,t)=&E[e^{-r(T-t)} \min(S_{T},F)\cdot {\bf 1}_{\{T<\tau_2\wedge\tau\}}\\
	\nonumber&+S_te^{-\delta(T-\tau)}e^{-r(\tau-t)}{\bf 1}_{\{\tau< T\wedge\tau_2\}} \\
	&+e^{-r(\tau_2-t)}\Phi_H(S_{\tau_2},\tau_2)\cdot {\bf 1}_{\{\tau_2< T\wedge\tau\}}\Big|\frac 1{ e^{-\delta (T-t)}}\Phi_L(S,t)<S_t=S<\frac 1{\gamma e^{-\delta (T-t)}}\Phi_L(S,t)\Big],
\end{align}
where ${\bf 1}_{\{event\}}=\left\{\begin{array}{ll}1,  &\mbox { if
		``event" happens},\\
	0, &\mbox{ otherwise. }\end{array}\right.$
\subsection{The PDE problem} In the life time of the bond, by Feynman-Kac formula
(see, e.g. \cite{DP}), it is not difficult to derive that the
letting values $\Phi_H$ and $\Phi_L$ satisfy the following system of partial
differential equations in their respective life regions:
\begin{eqnarray} \nonumber
	&&\frac{\partial\Phi_H}{\partial t} +\frac12 \s^2_HS^2\frac{\partial^2\Phi_H}{\partial S^2} + rS\frac{\partial\Phi_H}{\partial S}-r\Phi_H= 0,
	\\ &&\qquad\qquad \qquad\qquad\qquad\qquad S>\frac 1{\gamma e^{-\delta (T-t)}}\Phi_H,
	\;t>0, \label{1.1a}
	\\ \nonumber
	&& \frac{\partial\Phi_L}{\partial t} +\frac12 \s^2_LS^2\frac{\partial^2\Phi_L}{\partial S^2} + rS\frac{\partial\Phi_L}{\partial S}- r\Phi_L= 0, \\
	&&\qquad\qquad \qquad\qquad\qquad\qquad \frac 1{ e^{-\delta (T-t)}}\Phi_L<S<\frac 1{\gamma e^{-\delta (T-t)}}\Phi_L,
	\;t>0. \label{1.1b}
\end{eqnarray}
If the bond life last to
maturity, $\Phi_H$ and $\Phi_H$ satisfy the terminal conditions:
$$ 
\Phi_H(S,T)=\Phi_L(S,T)=\min\{S,F\}.
$$
Define the function $\Phi$ as 
\begin{equation*}
	\begin{split}
	\Phi(S,t)=\left\{\begin{split}
	&\Phi_H (S,t), \text{ in the high rating region;}\\
	&\Phi_L (S,t),	\text{ in the low rating region;}\\
	&Se^{-\delta(T-t)}, \text{ in the default region.}
	\\
	\end{split}
\right.
	\end{split}
\end{equation*}
Then, it satisfies the following variational form
$$
\min\Big\{  \frac{\partial \Phi}{\partial t} +\frac{1}{2} \sigma^2(\Phi,S,t)S^2 \frac{\partial^2 \Phi}{\partial S^2}+rS\frac{\partial \Phi}{\partial S}-r\Phi,\, -\Phi(S,t)+Se^{-\delta(T-t)} \Big\}=0,
$$
with 
$$
\sigma(\Phi,S,t)=\sigma_H{\bf 1}_{\{ \Phi<\gamma Se^{-\delta(T-t)} \}}+ \sigma_L{\bf 1}_{\{ \Phi\geqslant\gamma Se^{-\delta(T-t)} \}}.
$$
First, we make some transformation. Let $\phi(x,t)=e^{rt}\Phi(e^x,T-t)$. Then, $\phi$ satisfies
 $$
 \min\Big\{  -\frac{\partial \phi}{\partial t} +\frac{1}{2} \sigma^2(e^{-rt}\phi,e^x,t)\frac{\partial^2 \phi}{\partial x^2}+(r-\frac{1}{2}\sigma^2)\frac{\partial \phi}{\partial x},\, -\phi(s,t)+e^{x+(r-\delta)t} \Big\}=0.
 $$
 As already indicated in \cite{LWH16}, it is more convenient to work in the moving coordinate frame
 $$
 \xi=x+ct,\; c=r-\delta,\; u(\xi,t)=\phi(x,t).
 $$
 Then, the equation reads
 \begin{equation}\label{problem_evolution}
  \min\left\{  -\frac{\partial u}{\partial t} +\frac{1}{2} \sigma^2(u)\frac{\partial^2 u}{\partial \xi^2}+(\delta-\frac{1}{2}\sigma^2)\frac{\partial u}{\partial \xi},\,-u+e^{\xi} \right\}=0.
 \end{equation}
Let us introduce the weight $e^{-\xi}$ and make the further transformation
$ v=e^{-\xi} u$; we define
$${\mathcal L}:=-\frac{\partial }{\partial t} +\frac{1}{2} \sigma^2(v)\Big(\frac{\partial^2 }{\partial \xi^2}+\frac{\partial }{\partial \xi}\Big)+\delta\Big(\frac{\partial }{\partial \xi}+ 1\Big).$$
Thus, $v$ satisfies the following problem:
\begin{eqnarray}\label{problem_evolution1}
\min\left\{  {\mathcal L}v,\,1-v \right\}=0, \quad v(\xi,0)=\min\{1, e^{-\xi}\},
\end{eqnarray}
 with 
$$
 \sigma(v)=\sigma_H{\bf 1}_{\{ v<\gamma \}}+ \sigma_L{\bf 1}_{\{ v\geqslant\gamma \}}.
$$
Let us finally define the free boundaries which will play a crucial role throughout the paper, respectively the \textit{default boundary} 
$$\hat \kappa(t):=\inf \{ \xi\,|\, v(\xi,t)<1\},$$
and the \textit{transit boundary}
$$
\hat \eta(t):=\inf \{ \xi\,|\, v(\xi,t)<\gamma\}.
$$
Our goal is not only to solve \eqref{problem_evolution1} but also to study the properties of these boundaries. 
If the solution is smooth enough, system \eqref{problem_evolution1} can be rewritten as the \textit{free boundary problem}
\begin{equation}\label{FBP}
\left\{
\begin{split}
&-\frac{\partial v}{\partial t}+\frac12\sigma_L^2\Big(\frac{\partial^2 v}{\partial \xi^2}+\frac{\partial v}{\partial \xi} \Big)+\delta \Big(\frac{\partial v}{\partial \xi}+v \Big)=0, \quad \hat \kappa(t) <\xi<\hat \eta(t);\\
&-\frac{\partial v}{\partial t}+\frac12\sigma_H^2\Big(\frac{\partial^2 v}{\partial \xi^2}+\frac{\partial v}{\partial \xi} \Big)+\delta \Big(\frac{\partial v}{\partial \xi}+v \Big)=0, \quad \xi>\hat \eta(t);	\\
&v(\hat \kappa(t)+)=1,\quad \frac{\partial v}{\partial \xi}(\hat \kappa(t)+)=0;\\
&v(\hat \eta(t)+)=v(\hat \eta(t)-)=\gamma, \quad \frac{\partial v}{\partial \xi}(\hat \eta(t)+)=\frac{\partial v}{\partial \xi}(\hat \eta(t)-).
\end{split}\right.
\end{equation}
For convenience, we set
$$c_L= \frac{2\delta}{\sigma^2_L},\qquad c_H=\frac{2\delta}{\sigma^2_H}.$$
It follows from \eqref{main_sigma} that $c_L<1$ and $c_H>1$.
\section{Traveling Wave Solution}
  In this section, we will consider the steady state of \eqref{problem_evolution1}, i.e. the traveling wave solution for the original problem. In addition to giving the asymptotic behavior  of \eqref{problem_evolution1}, the traveling wave equation is also useful for constructing sub-solutions.  The traveling wave solution $K$ satisfies 
\begin{equation}\label{steady0}
	\min\left\{ \frac{1}{2} \sigma^2(K)\Big(\frac{dK}{d\xi^2}+\frac{dK}{d\xi}\Big)+\delta\Big (\frac{dK}{d\xi}+K\Big)  ,\,1-K \right\}=0.
\end{equation}
 Denoting the two free boundaries respectively by $\kappa^*$ and $\eta^*$, and assuming that the solution is sufficiently smooth, we may reformulate Equation \eqref{steady0} as  the following free boundary problem
 \begin{equation}\label{steady1}
 \left \{
 \begin{split}
 	&\frac{d^2 K}{d\xi^2}+\frac{dK}{d\xi}+c_H\Big(\frac{dK}{d\xi} +K\Big)=0,  \quad \xi>\eta^*,\\
 	&\frac{d^2 K}{d\xi^2}+\frac{dK}{d\xi}+c_L\Big(\frac{dK}{d\xi} +K\Big)=0,  \quad \kappa^* <\xi<\eta^*,\\
 	&K(\kappa^*+)=1, \quad \frac{\pr K}{\pr \xi}(\kappa^*)=0,\\
 	&K(\eta^*+)=K(\eta^*-)=\gamma, \quad \frac{dK}{d\xi}({\eta^*+})= \frac{dK}{d\xi}({\eta^*-}),\\
 	&K(\xi)=1, \text{for $\xi <\kappa^*$, and }\lim_{\xi \rightarrow +\infty} e^{\xi}K(\xi)=1, 
 \end{split}
 \right.
 \end{equation} 
Note that we also add a growth condition at $+\infty$ due to the financial nature of our problem. 
\begin{theorem}\label{thm_TW} 
System \eqref{steady1} has a unique solution $(K,\eta^*,\kappa^*)$ such that $K$ belongs to $C^1([\kappa^*,+\infty))$ and the respective restrictions of $K$ to $[\kappa^*,\eta^*]$ and $[\eta^*,+\infty]$ are $C^\infty$.
\end{theorem} 
\begin{proof}
	It is elementary to solve the second order system  in \eqref{steady1}:
	\begin{equation}\label{steady_sol}
	K(\xi)=\left\{
	\begin{split}
	&e^{-\xi}+Be^{-c_H\xi}, \;\xi>\eta^*,\\
	&Ce^{-\xi}+De^{-c_L\xi},\; \kappa^* <\xi<\eta^*.
	\end{split}
	\right.
	\end{equation}
	From the boundary conditions at $\kappa^*$, it comes 
	$$
	Ce^{-\kappa^*}+De^{-c_L\kappa^*}=1,\text{ and }-Ce^{-\kappa^*}-c_LDe^{-c_L\kappa^*}=0.
	$$
	This implies that $C=-\frac{c_L}{1-c_L}e^{\kappa^*}$ and $D=\frac{1}{1-c_L}e^{c_L \kappa^*}$.
	Then, from $ K(\eta^*-)=\gamma $, we have that 
	\begin{equation}\label{diff_equa}
	-\frac{c_L}{1-c_L}e^{\kappa^*-\eta^*}+\frac{1}{1-c_L}e^{-c_L(\eta^*-\kappa^*)}=\gamma.
	\end{equation}
	{Define the mapping 
		\begin{eqnarray}\label{psi}
		\Psi(x): x \mapsto -\frac{c_L}{1-c_L}e^{-x}+\frac{1}{1-c_L} e^{-c_Lx},
		\end{eqnarray}
		hence $\Psi'(x)=\frac{c_L}{1-c_L}(e^{-x}-e^{-c_L x})$. Since $c_L<1$, we have that the mapping $\Psi$ 
		is  decreasing on $[0,\infty)$. Since $\Psi(0)=1$ and $\lim_{x \rightarrow +\infty}\Psi(x)=0$,} the transcendental  equation \eqref{diff_equa} admits a unique positive solution
	\begin{equation}\label{eq_diff}
	\eta^*-\kappa^* = \Psi^{-1}(\gamma),
	\end{equation}
	The interface condition $\big[\frac{dK}{d\xi}\big]_{\eta^*}=0$ yields that 
	$$
	e^{-\eta^*} +c_H Be^{-c_H \eta^*}=-\frac{c_L}{1-c_L}e^{-(\eta^*-\kappa^*)}+\frac{c_L}{1-c_L} e^{-c_L(\eta^*-\kappa^*)}=\gamma -e^{-c_L(\eta^*-\kappa^*)},
	$$
	where the last equality is due to \eqref{eq_diff}. Combining with the condition $\gamma=K(\eta^*+)=e^{-\eta^*}+Be^{-c_H \eta^*}$, we have that 
	$$
	B=-\frac{1}{c_H-1} e^{-c_L(\eta^*-\kappa^*)+c_H\eta^*}\text{ and } (c_H-1)e^{-\eta^*}=(c_H-1)\gamma +e^{-c_L(\eta^*-\kappa^*)}.
	$$
	This implies that 
	\begin{equation}\label{eq_eta}
	\eta^*= -\log\left(\gamma + \frac{1}{c_H-1} e^{-c_L\Psi^{-1}(\gamma)}\right).
	\end{equation}
	Thus, $\kappa^*$, $B,C$ and $D$ are determined. Summarizing, it comes
		\begin{eqnarray}\label{K_equation}
		K(\xi)=\left\{\begin{array}{ll}
		e^{-\xi}+(\gamma-e^{-\eta^*})e^{-c_H(\xi-\eta^*)},  \;&\xi>\eta^*,\\[2mm]
		-\frac{c_L}{1-c_L}e^{-(\xi-\kappa^*)}+\frac{1}{1-c_L} e^{-c_L(\xi-\kappa^*)},\; &\kappa^* <\xi<\eta^*.
		\end{array}\right.
		\end{eqnarray}
\end{proof}

Some properties of $K$ are needed in the sections below. We list them in the following proposition.
\begin{proposition}\label{prop_TW}
(i) for $\xi > \kappa^*$, $\frac{dK}{d\xi}< 0$, $K+\frac{dK}{d\xi} >0$, and $\frac{d^2K}{d\xi^2}+\frac{dK}{d\xi} <0$ if $\xi \neq \eta^*$; \\
(ii) $\gamma <K(\xi)<1$ if $\xi \in (\kappa^*,\eta^*)$ and $K(\xi)<\gamma <1$ if $\xi >\eta^*$;\\
(iii) for $\xi \geqslant \kappa^*$, $K(\xi) \leqslant \min\{ 1,e^{-\xi}\}$;\\
(iv) $\eta^*$ is  a decreasing function of $\gamma$. Moreover, $\lim_{\gamma \rightarrow 0}\eta^*=+\infty$ and $\lim_{\gamma \rightarrow 1}\eta^*=-\log \frac{c_H}{c_H-1}$. 	
\end{proposition}
\begin{proof}
	(i) It is straightforward to compute
	$$\frac{dK}{d\xi}=\left\{
	\begin{array}{ll}
	-e^{-\xi}-c_H(\gamma-e^{-\eta^*})e^{-c_H(\xi-\eta^*)},  \;&\xi>\eta^*,\\[1mm]
	\frac{c_L}{1-c_L}e^{-(\xi-\kappa^*)}-\frac{c_L}{1-c_L} e^{-c_L(\xi-\kappa^*)},\; &\kappa^* <\xi<\eta^*.
	\end{array}
	\right.$$ 
	Since $c_L<1$, it holds that $\frac{dK}{d\xi} < 0$ for $\kappa^* <\xi<\eta^*$. For $\xi >\eta^*$, we rewrite $$\frac{d K}{d \xi}=-e^{-\eta^*}e^{-(\xi-\eta^*)}-c_H(\gamma-e^{-\eta^*})e^{-c_H(\xi-\eta^*)}.$$ With the notation from Theorem \ref{thm_TW}, we have that $$c_HBe^{-c_H\eta^*}= c_H(\gamma-e^{-\eta^*})$$ and 
	$$
	e^{-\eta^*} +c_H Be^{-c_H \eta^*}=-\frac{c_L}{1-c_L}e^{-(\eta^*-\kappa^*)}+\frac{c_L}{1-c_L} e^{-c_L(\eta^*-\kappa^*)}>0.
	$$
	Since $c_H>1$, it holds that $\frac{d K}{d \xi}< 0$ for $\xi>\eta^*$. Next, it comes 
	$$K+\frac{d K}{d \xi}=\left\{
	\begin{array}{ll}
	(1-c_H)(\gamma-e^{-\eta^*})e^{-c_H(\xi-\eta^*)},  \;&\xi>\eta^*,\\[2mm]
	e^{-c_L(\xi-\kappa^*)},\; &\kappa^* <\xi<\eta^*,
	\end{array}
	\right.$$
	and
	$$\frac{d K}{d \xi}+\frac{d^2 K}{d \xi^2}=\left\{
	\begin{array}{ll}
	(c_H^2-c_H)(\gamma-e^{-\eta^*})e^{-c_H(\xi-\eta^*)},  \;&\xi>\eta^*,\\[2mm]
	-c_Le^{-c_L(\xi-\kappa^*)},\; &\kappa^* <\xi<\eta^*.
	\end{array},
	\right.$$
	Noting that $e^{-\eta^*}=\gamma+\frac{1}{c_H-1} e^{-c_L\Psi^{-1}(\gamma)}>\gamma$, $c_L<1$ and $c_H>1$, we achieve the desired results.
	 
	(ii) It follows immediately from (i).
	
	(iii) We know from (ii) that  $K(\xi) \le 1$. On the one hand, thanks to (\ref{eq_eta}), $\gamma-e^{-\eta^*}<0$ hence $K(\xi)< e^{-\xi}$ if $\xi >\eta^*$  (see \eqref{K_equation}). On the other hand, note that $K+\frac{d K}{d \xi} > 0$ implies that $\xi \mapsto e^{\xi}K(\xi)$ is increasing, which indicates that $K(\xi) <e^{-\xi}$ for 
	$\kappa^* <\xi<\eta^*$. 
	
	(iv) Since $\Psi^{-1}$ is decreasing with respect to $\gamma$ and $c_H>1$, it follows from \eqref{eq_eta} that $\eta^*$ is  decreasing with respect to $\gamma$. It also holds that $\lim_{\gamma \rightarrow 0}\Psi^{-1}(\gamma)=+\infty$ and  $\lim_{\gamma \rightarrow 1}\Psi^{-1}(\gamma)=0$, hence the result.
\end{proof}

\section{Penalized and Regularized Cauchy Problem}

Problem \eqref{problem_evolution1} has singularities: at $v=\gamma$ due to the indicator function in the definition of $\sigma$; at $v=1$ as in a usual obstacle problem; and at $t=0$ because of the lack of regularity of the initial condition. To address these issues, we introduce $H_\e$, $\beta_\e$ and $\psi_\e$ which depend upon a small positive parameter $\e$. These smooth functions are chosen as the following. Let $H(s)$ be the Heaviside function, i.e., $H(s)=0$ for
$s<0$ and $H(s)=1$ for $s>0$. Then, $\sigma(v)$ in (\ref{problem_evolution1}) reads
 \[
 \s(v) = \s_H + (\s_L-\s_H) H(v - \gamma ).
\] 
First, we approximate $H$ by a $C^\infty $ function $H_\e$ such
that
\[
 H_\e(s) = 0 \,\mbox{ for } s<-\e, \, H_\e(s) =1 \mbox{ for } s>0, \, 0\leqslant H_\e'(s)\leqslant 2/\e\ \mbox{ for } -\infty<s<\infty.
\]
Second, let $\beta_\e(y)$ be  a smooth penalty function satisfying the following condition:
$$
\beta_\e(y) \in C^{\infty}(\mathbb R),\, \beta_\e(y) \geqslant 0,\, \beta_{\e}(y)=0\text{ if $y\leqslant -\e$;}
$$
$$
\beta_\e(0)=C_0\geqslant 2\delta;\; \beta'_\e(y) \geqslant 0; \;\beta''_\e(y)\geqslant 0;
$$
$$
\lim_{\e \rightarrow 0} \beta_\e(y)=0 \text{ if $y<0$; and }  \lim_{\e \rightarrow 0} \beta_\e(y)=+\infty \text{ if $y>0$.}
$$
Let $\varepsilon_\beta>0$ be the unique solution of $\beta_{\e}(-\frac{\e_{\beta}}{2})=\delta$. It is easy to see that $\e_\beta \rightarrow 0$ when $\e \rightarrow 0$.
Finally, let us define $\psi_{\e}(y):=1+\e_\beta\psi(\frac {y-1}{\e_\beta})$, where   $\psi \in C^{\infty}$, $\psi(y)=0$ for $y \geqslant 1/2$; $\psi(y)=y$ for $y<-1/2$ and $\psi(y)\leqslant y,\, 0 \leqslant \psi'(y) \leqslant 1,\, \psi''(y)\leqslant 0$ for $-1/2\leqslant y \leqslant 1/2$.

From the construction of $\psi_\e$, we have the following lemma.
\begin{lemma}\label{lem_psi}
	    (i) For $y \geqslant 0$, $0 \leqslant y\psi'_{\e}(y)\leqslant (1+\e_\beta)$; \;
		(ii) $0\leqslant  \psi_{\e}(y)-y\psi'_{\e}(y) \leqslant 1$.
\end{lemma}
\begin{proof}
(i) It is easy to see that $y\psi'_\e(y)=y\psi'(\frac{y-1}{\e_\beta})$, hence positive for $y\geqslant 0$. Note that $\psi'(\frac{y-1}{\e_\beta})=0 $ for $y \geqslant 1+\frac{\e_\beta}{2}$ and $\psi'(\frac{y-1}{\e_\beta}) \leqslant 1$ for $y \leqslant 1+\frac{\e_\beta}{2}$. Then, we shall have the second inequality. \\ 
(ii) Differentiating $\psi_{\e}(y)-y\psi'_{\e}(y)$, we have that 
\begin{equation}\label{ineq_derivative}
(\psi_{\e}(y)-y\psi'_{\e}(y))'=-\frac{y}{\e_\beta}\psi''(\frac{y-1}{\e_\beta}).
\end{equation}
This implies that the minimum is achieved at $y=0$. Thus,
$$
\psi_{\e}(y)-y\psi'_{\e}(y) \geqslant \psi_{\e}(0)=0.
$$
It is easy to verify that $\psi_{\e}(y)-y\psi'_{\e}(y)=1$ for $y<1-\frac{\e_\beta}{2}$ or $y>1+\frac{\e_\beta}{2}$. From \eqref{ineq_derivative}, we see that $\psi_{\e}(y)-y\psi'_{\e}(y) \leqslant 1$ for any $y$.
\end{proof}

Now, for $\e$ small, we consider the following approximated Cauchy problem:
\begin{equation}\label{2.1} 
\begin{split}
{ {\mathcal L}_{\e}}[v_{\e}] = -\frac{\pr v_{\e}}{\pr t}
+\frac12 \s^2_{\e}(v_\e)\Big(\frac{\pr^2v_{\e}}{\pr \xi^2} + \frac{\pr v_{\e}}{\pr \xi}\Big)+\delta\Big(\frac{\pr v_{\e}}{\pr \xi}
+v_{\e}\Big) =\beta_\e(v_{\e}-1),
\end{split} 
\end{equation}
where $(\xi,t) \in \Omega_T =\mathbb R \times (0,T]$, $T>0$, and
\begin{equation}\label{sigma}\s_\e(v_\e) = \s_H + (\s_L-\s_H)
H_\e(v_{\e} - \gamma), \end{equation} 
together with the initial condition
 \begin{equation}\label{2.2}
 v_{\e}(\xi,0) = \psi_{\e}(e^{-\xi}), \quad \xi\in\mathbb R.
\end{equation} 
Hence, from the definition of $\psi_\e$ in the previous, we have that $v_\e(\xi,0)=1$ for $\xi \leqslant -\log (1+\frac{\e_\beta}{2})$; $v_\e(\xi,0)=e^{-\xi}$ for $\xi \geqslant -\log(1-\frac{\e_\beta}{2})$. We have the following existence result: 
\begin{theorem}\label{existence_approximated}
	For $\e>0$ fixed, problem \eqref{2.1}-\eqref{2.2} has a unique bounded classical solution $v_\e$. Moreover, $v_\e \in C^\infty(\mathbb R \times[0,T])$.
\end{theorem}
\begin{proof}
	First, we turn Equation \eqref{2.1} into a quasilinear equation whose principal part is in divergence form:
		\begin{equation}
		\frac{\pr v_{\e}}{\pr t}-\frac{\partial }{\pr \xi} a\big(\xi,v_\e,\frac{\partial v_\e }{\pr \xi}\big)+A\big(\xi,v_\e,\frac{\partial v_\e }{\pr \xi} \big)=0,
		\end{equation}
		with 
		$$
		a(\xi,v,p)=\frac{1}{2}\sigma^2_\e(v) p, \quad A(\xi,v,p)=\beta_\e(v-1)-\delta v-\big(\frac{1}{2}\sigma^2_\e(v)+\delta\big)p+\sigma_\e\sigma'_\e(v)p^2.
		$$
One can check that  $a$ and $A$ satisfy the assumptions of \cite[Chapter V, Theorem 8.1]{La}. Thus, there exists a unique bounded solution $v_\e \in C^{2+\alpha,1+\frac\alpha2}(\mathbb R \times[0,T])$ for any $0<\alpha<1$.\footnote{For usual parabolic H\"older spaces, see, e.g., \cite[Chapter 1, Section 1]{La},\cite[Section 5.1]{Lunardi96}).} Then, $\sigma_\e(v_\e)$ and $\beta_\e(v_\e)$ belong to the same function class. Further H\"older regularity follows from classical results for linear problems (see \cite[Chapter IV, Theorem 5.1]{La}, \cite[Theorem 5.1.10]{Lunardi96}), which yields that  $v_\e \in C^{4+\alpha,2+\frac\alpha2}(\mathbb R \times[0,T])$. The result follows by bootstrapping.
\end{proof}
\begin{remark}\label{r1}
From the definition of $H_\e$ and $\beta_\e$, it is easy to see that $\sigma_\e(v_\e)=\sigma_L$ when $v_\e >\gamma$ and $\beta_\e(v_\e)=0$ when $v_\e<1-\e$. Thus, when $\e$ is small enough,  at least one of these two equations holds.
\end{remark}

\subsection{Estimates on the approximating solution}
We now proceed to derive necessary estimates on $v_{\e}$ independent of $\e$, via the the maximum principle for parabolic equations in unbounded domains (see, e.g., \cite[Chapter 2]{F3}, \cite[Chapter 7]{PAO}). These properties will be inherited by the limit $v$ when taking $\e \to 0$ and, thus, are crucial for the analysis of the bond value and free boundaries.

\begin{lemma}\label{l1a} For $\e$ sufficiently small,  it holds in $\mathbb R \times[0,T]$:
\[
 0\leqslant  v_{\e}\leqslant \min (1, e^{-\xi}).
\]
\end{lemma}
\begin{proof}
Recall that we have introduced a smooth cut-off function $\psi$ in the beginning of this section. Define a function $h$ as $h(y):= \e \psi(\frac{y-\frac12}{\e})+\frac12$. Then, we see that $h(y)=\frac12$ for $y \geqslant \frac12(1+\e)$; $h(y)=y$ for $y\leqslant \frac12(1-\e)$ and $ 0 \leqslant h'(y) \leqslant 1,\, h''(y)\leqslant 0$ for $y \in \mathbb R$.  Thus, it holds that $h(y)\geqslant 0$ if and only if $y \geqslant 0$. Furthermore, one can directly check that $|\frac{yh'(y)}{h(y)}|$ is bounded. Let $w=h(v_\e)$ and we have that 
$$
\mathcal L_\e[w]=h'(v_\e)\beta_\e(v_\e-1)+\frac12 \sigma^2_\e(v_\e)h''(v_\e)(\frac{\partial v_\e}{\partial \xi})^2+\delta(w-h'(v_\e)v_\e),
$$
which can be rewritten as 
$$
\mathcal L_\e[w]-\delta(1-\frac{v_\e h'(v_\e)}{h(v_\e)})w=h'(v_\e)\beta_\e(v_\e-1)+\frac12 \sigma^2_\e(v_\e)h''(v_\e)(\frac{\partial v_\e}{\partial \xi})^2.
$$
Since $\beta_\e(v_\e-1)=0$ when $v_\e<1-\e$ and $h'(v_\e)=0$ when $v_\e>\frac12(1+\e)$, we see that $h'(v_\e)\beta_\e(v_\e-1)=0$ if $\e$ is sufficiently small. Noting that $h''\leqslant 0$, it holds that $\mathcal L_\e[w]-\delta(1-\frac{v_\e h'(v_\e)}{h(v_\e)})w \leqslant 0$. As the coefficient of zeroth order term is bounded, one can apply maximum principle to get that $w \geqslant 0$, which is equivalent to $v_\e\geqslant0$.

Next, set $w=v_\e -1$. Then, $w$ verifies 
\begin{equation*}
\begin{split}
\mathcal L_{\e} [w]&=\beta_{\e}(w)-\delta=\frac{\beta_{\e}(w)-\beta_{\e}(0)}{w}w+\beta_{\e}(0)-\delta.
\end{split}
\end{equation*}
From the definition of $\beta_{\e}$, it holds that $\beta_\e(0)=C_0\geqslant 2\delta $. Hence, this leads to $w\leqslant 0$ according again to the maximum principle.	

Finally, Let $w=v_\e -e^{-\xi}$. Then, it holds that 
$$
\mathcal L_\e[w]=\beta_{\e}(v_\e-1)=\frac{\beta_\e(v_\e-1)-\beta_\e(e^{-\xi}-1)}{w}w +\beta_\e(e^{-\xi}-1).
$$
Noting that $w(\xi,0)\leqslant 0$ and $\beta_{\e}(e^{-\xi}-1)\geqslant 0$, we deduce that $w\leqslant 0$ according to the maximum principle.	

\end{proof}

\begin{lemma}\label{l2} It holds in $\Omega_T$:
\[
 -(1+\e_\beta)e^{\delta t}\leqslant  \frac{\pr v_{\e}}{\pr \xi} < 0.
\]
\end{lemma}
\begin{proof}
Differentiating \eqref{2.1}, it comes
$$
{\mathcal L}_\e\Big[ \frac{\pr v_{\e}}{\pr \xi}
\Big]=-\s_\e(v_\e)\,\s'_\e(v_\e)\frac{\pr v_{\e}}{\pr \xi} \,( \frac{\pr^2v_{\e}}{\pr \xi^2}+ \frac{\pr v_{\e}}{\pr \xi})+\beta'_\e(v_\e -1)\frac{\pr v_{\e}}{\pr \xi}.
$$
At $t=0,$ $ \frac{\pr v_{\e}}{\pr \xi}= {\color{black} -e^{-\xi}\psi'_{\e}(e^{-\xi})}$, which lies between $-(1+\e_\beta)$ and $0$ from the proof of Lemma \ref{lem_psi}. By the maximum principle, one can deduce that $ -(1+\e_\beta)e^{\delta t}\leqslant  \frac{\pr v_{\e}}{\pr \xi} \leqslant 0$. Furthermore, the strict inequality in $\Omega_T$ holds due to strong maximum principle.   		
\end{proof}

\begin{lemma}\label{l2a}  It holds in $\Omega_T$:
\[
 1\geqslant \frac{\pr v_{\e}}{\pr \xi} + v_{\e} > 0.
\]
\end{lemma}
\begin{proof}
By Lemma \ref{l1a} and \ref{l2}, we have the first inequality of the lemma. Then,
set $w = \frac{\pr v_{\e}}{\pr \xi} + v_{\e}$, $w(\xi,0)=-e^{-\xi}\psi'_\e(e^{-\xi})+\psi_\e(e^{-\xi})$. It follows from Lemma \ref{lem_psi} that $w \geqslant 0$ at $t=0$.  Also, $w$ verifies
\begin{eqnarray}\label{ww1}
{\mathcal L}_\e[w]  +
  \s_\e (v_{\e}) 
 \s_\e'(v_{\e})\, \frac{\pr v_{\e}}{\pr \xi} \frac{\pr  w}{\pr \xi}  = \beta_\e'(v_{\e}-1) (w-v_{\e})+\beta_\e (v_{\e} -1).
\end{eqnarray}
Using Taylor expansion of $\beta_\e(-\e)$ at $y$, one has that
$$ 0 =\beta_\e(-\e) =\beta_\e(y)-\beta'_\e(y)(y+\e)+\frac
1{2}\beta_\e''(\theta)(y+\e)^2.$$
That is,
$$\beta_\e(y)-(y+\e)\beta_\e'(y)\leqslant 0.$$
Replacing $y$ by $v_{\e}(\xi)-1$  in the above formula, we have,
$$\beta_\e(v_\e-1)-(v_{\e}-1+\e)\beta_\e'(v_\e-1)\leqslant 0.$$
Thus, \eqref{ww1} reads
\begin{eqnarray*}\label{ww0}
&& -{\mathcal L}_\e[w]  -
 \s_\e(v_\e) 
 \s_\e'(v_\e)\,\frac{\pr v_{\e}}{\pr \xi} \frac{\pr w}{\pr \xi}+\beta_\e'(v_\e-1) w\\
 &&=\beta_\e'(v_\e-1) v_{\e}-\beta_\e(v_\e-1)\geqslant
  v_{\e}\beta_\e'(v_\e-1)-(v_{\e}-1+\e)\beta_\e'(v_\e-1)=(1-\e)\beta_\e'(v_\e-1)\geqslant
 0.
\end{eqnarray*}
By the strong maximum principle, $w>0$ in $\Omega_T$.
 \end{proof}

	\begin{lemma}\label{l3b} For sufficiently small $\e$,  it holds in $\mathbb R \times[0,T]$:
		\[
		\frac{\pr^2 v_{\e}}{\pr \xi^2} + \frac{\pr v_{\e}}{\pr \xi} \leqslant 0.
		\]
	\end{lemma}
	\begin{proof}
		At $t=0,$
		$\frac{\pr^2 v_{\e}}{\pr \xi^2} + \frac{\pr v_{\e}}{\pr \xi}=e^{-\xi}\psi''_\e(e^{-\xi})$ is non-positive. Now, consider the function $w = \frac{\pr v_{\e}}{\pr t}-\delta\big(\frac{\pr v_{\e}}{\pr \xi} + v_{\e}\big)+\beta_\e(v_\e-1)$. From the definition of $H_\e$ and $\beta_\e$, it is easy to see that $\sigma_\e(v_\e)=\sigma_L$ when $v_\e >\gamma$ and $\dyc{\beta_\e(v_\e-1)}=0$ when $v_\e<1-\e$. Thus, we divide the space into two parts $\{v_\e <1-2\e\}$ and $\{ v_\e\geqslant 1-2\e\}$\\[1mm]
			\noindent{\bf Case 1: $v_\e <1-2\e$. }
			In this case, we see that $\beta_\e \equiv 0$. Then, it holds that 
			\begin{eqnarray}\nonumber
				&&{\mathcal L}_\e[w] + \s_\e(v_\e) 
				\s_\e'(v_\e) \Big (\frac{\pr^2 v_{\e}}{\pr \xi^2}+ \frac{\pr v_{\e}}{\pr \xi}\Big) \Big(\frac{\pr v_{\e}}{\pr t} -\delta \frac{\pr v_{\e}}{\pr \xi}\Big) \\
				&&= {\mathcal L}_\e[w]+\frac{2\s_\e'(v_\e)}{\s_\e(v_\e)} ( w +\delta v_{\e}) w =0. \label{c1}
			\end{eqnarray}
			\noindent{\bf Case 2: $ v_\e\geqslant 1-2\e$.}   For sufficiently small $\e$, we have that $\sigma_\e\equiv \sigma_L$. Then, it holds that 
			$$
			\mathcal L_\e[\frac{\pr v_\e}{\pr t}]=\beta_\e' \frac{\pr v_\e}{\pr t},\mathcal L_\e[\frac{\pr v_\e}{\pr \xi}]=\beta_\e' \frac{\pr v_\e}{\pr \xi},
			$$
			and
			$$
			\mathcal L_\e[\beta_\e]=\beta_\e'\beta_\e+\frac12 \sigma^2_\e\beta_\e''\left(\frac{\pr v_\e}{\pr \xi}\right)^2+\delta \beta_\e-\delta \beta'_\e v_\e,
			$$
			where  we denote $\beta_{\e}= \beta(v_\e -1)$, $\beta'_{\e}= \beta'(v_\e -1)$ and $\beta''_{\e}= \beta''(v_\e -1)$ for simplicity of the notations. Combining above equations, we have that 
			$$
			\mathcal L_\e[w]-\beta_\e' w=\frac12 \sigma^2_\e\beta_\e''\left(\frac{\pr v_\e}{\pr \xi}\right)^2\geqslant0,
			$$
			where the last inequality is due to the fact that $\beta_\e$ is convex.

			Combining these two cases, it holds that 
			$$
			\mathcal L_\e[w]+\Big( \frac{2\s_\e'(v_\e)}{\s_\e(v_\e)} ( w +\delta v_{\e})1_{\{ v<1-2\e\}}-\beta_\e' 1_{\{v\geqslant1-2\e\}} \Big)w \geqslant 0.
			$$
			Then, by the maximum principle, $w\leqslant 0$.
		
	\end{proof}

\begin{lemma}\label{l3} It holds in $\mathbb R \times(0,T]$:
\[
 \frac{\pr v_{\e}}{\pr t} < 0.
\]
\end{lemma}

\begin{proof}
Set $w=\frac{\pr v_\e}{\pr t}$. Then, we see that 
\begin{equation} \label{eq_deri_t}
	\mathcal L_\e[w]=-\sigma_\e(v_e) \sigma'_\e(v_e) (\frac{\pr^2 v_\e}{\pr \xi^2}+\frac{\pr v_\e}{\pr \xi})w+\beta'_\e(v_\e -1) w.
\end{equation}
Because $v_\e(\xi,0)=\psi_\e(e^{-\xi})$, we have that 
\begin{equation}
\begin{split}
w(\xi,0)=
\frac{1}{2}\sigma^2_\e(v_\e) e^{-2\xi}\psi''_\e(e^{-\xi})+\delta\big(\psi_{\e}(e^{-\xi})
-e^{-\xi}\psi'_\e(e^{-\xi})\big)-\beta_\e(\psi_{\e}(e^{-\xi})-1).
\end{split}
\end{equation}\label{initial_time_derivative}Since $\psi''_\e(\cdot)\leqslant 0$, the first term is negative. Then, it is easy to check that when $e^{-\xi}<1-\frac{\e_\beta}{2}$, the second term is zero. Hence, $w(\xi,0)\leqslant 0$ in this case. Now, it remains only to check the case $e^{-\xi} \geqslant 1-\frac{\e_\beta}{2}$. From Lemma \ref{lem_psi} and monotonicity of $\beta_\e$, we have that 
$$
\delta(\psi_{\e}(e^{-\xi})-e^{-\xi}\psi'_\e(e^{-\xi}))-\beta_\e(\psi_{\e}(e^{-\xi})-1)\leqslant \delta -\beta_\e(-\frac{\e_\beta}{2}).
$$
According to our choice of $\e_\beta$, we see that the above term is non-positive. Thus, we proved that $w(\xi,0) \leqslant 0$, which yields the desired result thanks to the strong maximum principle.
\end{proof}

\begin{lemma}\label{l3a} There are positive constants $c_1,C_2$ and $C_3$, independent of $\e$, such that it holds in $\mathbb R \times(0,T]$ 
\[
 \frac{\pr v_{\e}}{\pr t} \geqslant -C_3-\frac{C_2}{\sqrt{t}}\exp\Big(
  -c_1\frac{\xi^2}{t}\Big).
\]

\end{lemma}
\begin{proof}
Since $v_{\e}(0,0) = 1 >\gamma$, and by H\"older continuity of
the solution (see Theorem \ref{existence_approximated}), there exists a $\rho>0$,
independent of $\e$, such that
\[
v_{\e}(x,t) > (1+\gamma)/2\; \mbox{ in } B_\rho,
\]
where
$$B_\rho =\left\{(\xi,t),\, |\xi|\leqslant \rho,\; 0\leqslant t\leqslant \rho^2 \right\}.$$
Thus, for $\e$ small enough such that $\e<(1-\gamma)/2$, $\s_\e \equiv \s_L$ in $B_\rho$. 
We observe that, in $B_\rho$, the problem is reminiscent of a vanilla
American option, which has a lower estimate (see, e.g., \cite{LHJ})
\ben\label{2.5}
  \frac{\pr v_{\e}}{\pr t} \geqslant -C_2-\frac{C_2}{\sqrt{t}}\exp\Big(
  -c_1\frac{\xi^2}{t}\Big)
  \mbox{ in }B_\rho.
\een
Let us refer to Lemma \ref{l3} for the notation $w = \frac{\pr v_\e}{\pr t}$. From \eqref{initial_time_derivative}, it is easy to verify that $w(\xi,0)$ is uniformly bounded from below on $|\xi|\geqslant\rho$. Combining with \eqref{2.5}, there exists $C_3>0$ such that $w(x,t)\geqslant -C_3$ on $\{|\xi|\geqslant \rho, t=0 \} \cup \{ |\xi|=\rho, 0\leqslant t \leqslant \rho^2\} \cup \{|\xi|\leqslant \rho,t=\rho^2 \}$. The Maximum Principle (see Lemma \ref{l3}) yields that $w(\xi,t)\geqslant -C_3$ in $\Omega_T\setminus B_\rho$. Together with \eqref{2.5}, we get the desired result.
\end{proof}
As an immediate corollary, we have
\begin{lemma}\label{l4} There are positive constants $C_4, C_5$ and $C_6$, independent of $\e$, such that it holds in $\mathbb R \times(0,T]$ 
\[
  -C_4-\frac{C_5}{\sqrt{t}}\exp\Big( -c_1\frac{\xi^2}{t}\Big) \leqslant \frac{\pr^2 v_\e}{\pr \xi^2}\leqslant  C_6.
\]
\end{lemma}
\subsection{The approximating transit boundary}\label{approx_transit}
Let us denote by $\eta_\e(t)$  the approximating transit boundary, which is
implicitely defined by the equation
\ben\label{af}v_{\e}(\eta_{\e}(t),t)=\gamma.\een
We will construct the curve $t \mapsto \eta_\e(t)$ via the Implicit Function Theorem.
To begin with, we give a lower bound for $v_\e$.
Fom Lemma \ref{l1a}, it holds that $v_\e({\xi},t) \leqslant \gamma-\e$ when $\xi  \geqslant \log \frac{1}{\gamma-\e}$. This implies that $\sigma_\e=\sigma_H$ when $\xi \geqslant \log \frac{1}{\gamma-\e}$ and $\eta_\e(t) \leqslant \log \frac{1}{\gamma}$. Then, we give a lower bound for $v_\e$.
\begin{lemma}\label{lem_lower_bound}
	Let $(\tilde K,{\color{black}\tilde \eta^*},\tilde \kappa^*)$ be the solution of \eqref{steady1} as constructed in Theorem \ref{thm_TW} with $\gamma$ replaced by $\tilde \gamma$. Choose $\tilde \gamma$ properly such that $\tilde \eta^*=\log \frac2 {{\gamma}}$. Then, we have that $v_\e \geqslant \tilde K-(\e\vee \e_\beta) e^{\delta t}$ when $\e <\frac\gamma2$.
\end{lemma}
\begin{proof}
	From Proposition \ref{prop_TW} (iv), $\tilde \gamma$ is well defined. We can rewrite that 
	$$
	\frac{1}{2}\sigma_2^2\Big(\frac{d^2 \tilde K}{d \xi^2}+\frac{d \tilde K }{d \xi}\Big)+\delta\Big(\frac{d\tilde K}{d \xi}+\tilde K\Big)=\delta 1_{\{ \xi \leqslant \tilde \kappa^*\}}, 
	$$	
	with $\sigma_2:=\sigma_H 1_{\{ \xi \geqslant \tilde \eta^*\}}+\sigma_L 1_{\{ \xi < \tilde \eta^*\}}$. Let $w=v_\e -(\tilde K-(\e\vee \e_\beta) e^{\delta t})$. Then, it holds that 
	\begin{equation*}
	\begin{split}
	{\mathcal L}_{\e}[w]=\beta_\e({v_\e-1})- \delta 1_{\{ \xi \leqslant \tilde \kappa^*\}}+\frac{1}{2}(\sigma_2^2-\sigma_\e^2)\Big(\frac{d^2 \tilde K}{d \xi^2}+
	\frac{d \tilde K}{d \xi }\Big).
	\end{split}
	\end{equation*}
	Since we choose $\tilde \eta^*=\log \frac{2}{{\gamma}}$, it holds that $\sigma_\e^2 \leqslant \sigma^2_2 $. Combining with the fact that $\frac{d^2 \tilde K}{d \xi^2}+
	\frac{d \tilde K}{d \xi }\leqslant 0 $, we see that the last term on the right hand side is non-positive.
	Since $\tilde K\leqslant \min\{1,e^{-\xi}\}$ (see Proposition \ref{prop_TW} (iii)), 
	$\beta_\e(\tilde K-(\e\vee \e_\beta) e^{\delta t}-1)\leqslant \beta_\e(-\e)=0 $. Thus, 
	$$
	\mathcal L_\e[w]-\frac{\beta_\e(v_\e-1)-\beta_\e(\tilde K-\e\vee\e_\beta-1)}{w}w \leqslant 0.
	$$
	At $t=0$, $v_\e(\xi,0)=\psi_{\e}(e^{-\xi})$. It is easy to see that $\psi_\e(e^{-\xi})=e^{-\xi}$ for $e^{-\xi} \leqslant 1-\frac{\e_\beta}{2}$ and $\psi_\e(e^{-\xi}) \geqslant 1-\frac{\e_\beta}{2}$ for $e^{-\xi} \geqslant 1-\frac{\e_\beta}{2} $. For both cases, we have $v_\e(\xi,0)\geqslant \tilde K(\xi)-\e \vee \e_\beta$. The desired result follows from the maximum principle. 
\end{proof}

\begin{theorem}\label{IFT}
For fixed $\e>0$, there exists an decreasing smooth function $\eta_{\e}(t)$ such that 
\ben\label{initial_eta}
\eta_{\e}(0)= \log(\frac{1}{\gamma}), \quad \tilde \kappa^* < \eta_\e(t) \leqslant \log \frac{1}{\gamma},
\een
and \eqref{af} holds for all $t\in[0,T]$.
\end{theorem}
\begin{proof}
	To begin with, we compute 
	$$v_{\e}(-\log \gamma,0) = \psi_{\e}(e^{\log \gamma})= \psi_{\e}(\gamma)= 1 +\e_{\beta}\psi \big(\frac{\gamma-1}{\e_{\beta}}\big).
	$$
	Because $\gamma-1<0$, it is clear that $\frac{\gamma-1}{\e_{\beta}}< -\frac{1}{2}$ if $\e_{\beta}$ small enough, hence $\psi \big(\frac{\gamma-1}{\e_{\beta}}\big)=
\frac{\gamma-1}{\e_{\beta}}$ and $\psi_{\e}(\gamma)=\gamma$. Therefore, $v_{\e}(-\log \gamma,0) =\gamma$. We remind that the function $\xi \mapsto v_{\e}(\xi,0)$ is smooth and non-increasing; however, in some neighborhood of $-\log \gamma$ such that  $\frac{\gamma-1}{\e_{\beta}}< -\frac{1}{2}$, the function $v_{\e}(\xi,0)$ is decreasing which yields that the initial position of $\eta_{\e}$ is well-defined by \eqref{initial_eta}.

Next, we compute 
$$
\frac{\partial v_{\e}}{\partial \xi}(-\log \gamma,0)= -\gamma \psi_{\e}'(\gamma)=-\gamma <0,
$$
and (see the proof of Lemma \ref{l3})
$$
\frac{\partial v_{\e}}{\partial t}(-\log \gamma,0)= -\beta_{\e}(\gamma -1) =0.
$$
It is now an exercise to apply the Implicit Function Theorem, which shows that there exist $\delta_i,\tau_i>0, i=1,2$, and a unique function $\varphi_{\e} \in C^{\infty}([-\tau_1,\tau_2])$ such that, if $(\xi,t) \in [-\log \gamma-\delta_1,-\log \gamma+\delta_2] \times [-\tau_1,\tau_2]$ verifies $v_{\e}(\xi,t)=\gamma$, then $\xi=\varphi_{\e}(t)$. Note that  $\varphi_{\e}$ is a decreasing function because
$$
\varphi'_{\e}(t)= - \frac{\partial v_{\e}}{\partial t}(\varphi_\e(t),t)\,
\Big(\frac{\partial v_{\e}}{\partial \xi}(-\varphi_\e(t),t)\Big)^{-1} < 0.
$$
As  by product, taking the restriction of $\varphi_{\e}$ to $[0, \tau_2]$, we have constructed a (small) branch of the curve $\eta_{\e}$, of class $C^{\infty}$, such that \eqref{af} holds for all $t\in[0,\tau_2]$, $\eta_{\e}(0)=\gamma$.  Lemma \ref{lem_lower_bound} implies that $v_\e \geqslant 1-(\e\vee \e_\beta) e^{\delta t}$ for $\xi \leqslant \tilde \kappa^*$. Combining with the fact that $\frac{\pr v_\e}{\pr \xi} <0$, we see that  $\tilde \kappa^* < \eta_\e(t) \leqslant \log \frac{1}{\gamma}$.
 
 In view of Lemmas \ref{l2}  and \ref{l3}, we may reiterate the Implicit Function Theorem and continue this branch up to a endpoint achieved at time $T$.
\end{proof}

\begin{lemma}\label{lem_Lip}
	For any $T>0$, there exists a constant $C_T>0$, independent of $\e$, such that $\sup_{t\in[0,T]}|\eta'_\e(t)| \leqslant C_{T}$.
\end{lemma}
\begin{proof} From the Implicit Function Theorem, it holds:
	$$
	\eta'_{\e}(t)= - \frac{\partial v_{\e}}{\partial t}(\eta_{\e}(t),t)\,
	\Big(\frac{\partial v_{\e}}{\partial \xi}((\eta_{\e}(t),t)\Big)^{-1}.
	$$
	Note that Lemmas \ref{l3} and \ref{l3a} implies that $\frac{\pr v_\e}{\pr t}$ is bounded. To prove the desired results, we only need to show that $\frac{\pr v_\e}{\pr \xi}(\eta_\e(t),t) \leqslant -c$, for some positive $c$. In Lemma \ref{l3b}, we proved that $\frac{\pr^2 v_\e}{\pr \xi^2}+\frac{\pr v_\e}{\pr \xi} \leqslant 0$, which implies that $e^{\xi}\frac{\pr v_\e}{\pr \xi} $ is non-increasing in $\xi$. Since $v_\e$ is smooth, there exists a point $\hat \eta_\e(t) \in (\tilde \kappa^*, \eta_\e(t))$ such that
 	$$
	\frac{\pr v_\e}{\pr \xi}(\hat \eta_\e(t),t)=\frac{v_\e(\tilde \kappa^*,t)-v_\e(\eta_\e(t),t)}{\tilde \kappa^*-\eta_\e(t)}=-\frac{v_\e(\tilde \kappa^*,t)-v_\e(\eta_\e(t),t)}{\eta_\e(t)-\tilde \kappa^*}.
	$$
	We have shown that $v_\e(\tilde \kappa^*,t)\geqslant \tilde K(\tilde \kappa^*)-(\e\vee \e_\beta) e^{\delta t}=1-(\e\vee \e_\beta) e^{\delta t}$ and $\eta_\e(t) \leqslant \log \frac{1}{\gamma}$. This yields that 
	$$
	\frac{\pr v_\e}{\pr \xi}(\hat \eta_\e(t)) \leqslant -\frac{1-(\e\vee \e_\beta) e^{\delta t}-\gamma}{\log\frac1\gamma-\tilde \kappa^*}.
	$$
	Since $e^{\xi}\frac{\pr v_\e}{\pr \xi} $ is non-increasing, it holds that 
	$$
	\frac{\pr v_\e}{\pr \xi}(\eta_\e(t),t) \leqslant- e^{\hat \eta_\e(t)-\eta_\e(t)} \frac{1-(\e\vee \e_\beta) e^{\delta t}-\gamma}{\log\frac1\gamma-\tilde \kappa^*} \leqslant - e^{\tilde \kappa^*-\log\frac1\gamma} \frac{1-(\e\vee \e_\beta) e^{\delta t}-\gamma}{\log\frac1\gamma-\tilde \kappa^*} .
	$$
	This completes the proof.
\end{proof}
From Theorem \ref{IFT} and Lemma \ref{lem_Lip}, we see that the sequence $(\eta_\e)_{\e>0}$ is bounded in $C^1([0,T])$, therefore, extracting a subsequence if necessary, it converges uniformly to a function $\hat \eta(t)$.
\begin{cor} \label{hat_eta}
Extracting a subsequence if necessary, the sequence $\eta_{\e}$ converges uniformly to a limit $\hat \eta(t)$.
\end{cor}

\section{Main Results}
\subsection{Existence and Uniqueness}
Lemmas \ref{l1a}-\ref{l3b}   provide estimates  on the approximated solution $v_{\e}$. By taking a limit as $\e\to 0$, we are able to  derive the existence of a solution to \eqref{problem_evolution1}-\eqref{FBP}. 
\begin{theorem}\label{thm_exist}  (i) For any $T>0$, there exists a sequence $\e \to 0$ such that $v_\e \to v$ a.e. in $\mathbb R \times [0,T]$, $\frac{\partial v_\e}{\partial \xi} \to \frac{\partial v}{\partial \xi}$ a.e. in $\mathbb R \times [0,T]$, $v_\e \to v$ in $W^{1,0}_\infty(\mathbb R\times [0,T])$ weak-$*$ and $W^{2,1}_\infty((\mathbb R\times [0,T] )\setminus \overline{Q}_{\rho}) $\,weak-$*$, for any $\rho>0$, where $Q_\rho=(-\rho,\rho)\times (0,\rho^2)$. Moreover, extracting a subsequence if necessary, $\eta_\e$ converges uniformly to $\hat \eta$;\footnote{For $\Omega \subset \R\times [0,T]$, $W^{2,1}_{p}(\Omega)$, $1<p<\infty$, is the space of elements of $L^p(\Omega)$ whose derivatives are also in $L^p(\Omega)$, respectively up to second order in $\xi$ and to first order in $t$. $W^{2,1}_{\infty}(\Omega)$ is the space of bounded functions whose derivatives are bounded, respectively up to second order in $\xi$ and first order in $t$. $W^{1,0}_{\infty}(\Omega)$ denotes the space of bounded functions whose derivative w.r.t. $\xi$ is also bounded.}\\ 
(ii) $v$ is a solution of the original free boundary problem \eqref{problem_evolution1};\\
(iii) $v$ satisfies the estimates of Lemmas \ref{l1a}-\ref{l3b}, and the inequality
\begin{equation}\label{inequality}
\frac{\pr^2 v}{\pr \xi^2}+\frac{\pr v}{\pr \xi}\leqslant 0 \;\, \mbox{a.e. in} \;\, \mathbb R\times [0,T] )\setminus \overline{Q}_{\rho}, 
\end{equation}
as well as the following growth condition: there exists a constant $B>0$ such that $v(\xi,t)= 1$ when  $\xi < -B $ and $v(\xi,t)\leqslant e^{-\xi}$ when $\xi >B$, $0\leqslant t\leqslant T$.
\end{theorem}
\begin{proof} Let $(\e_n)_{n\geqslant 1}$ be a sequence converging to $0$ when $n \to +\infty$ and consider the corresponding solutions $(v_{\e_{n}})$ of \eqref{2.1} and \eqref{2.2}. For simplicity, we denote $v_{\e_{n}}$ by $v_n$. According to  Lemmas \ref{l1a}-\ref{l4},  we first observe that the sequence $(v_n)$ is bounded in the spaces $W^{1,0}_\infty(\mathbb R\times [0,T] )\cap W^{2,1}_\infty((\mathbb R\times [0,T] )\setminus \bar Q_{\rho} )$. Second, the sequence $(v_n)$ is bounded in the space $W^{2,1}_{p,\mbox{\tiny{loc}}}(\mathbb R\times[0,T])$, $1<p<2$.

Next, let $(A_m)_{m\geqslant 1}$ be a sequence of positive numbers such that $A_m \to +\infty$ as $m \to +\infty$. Let us consider the restriction $v_{n}^m$ of $v_n$ to the interval $[-A_m,A_m]$. At fixed $m\geqslant 1$, the sequence $(v_{n}^m)$ is bounded in the space $W^{2,1}_{p}([-A_m,A_m])\times [0,T])$ for any $1<p<2$. One can extract a subsequence, denoted by $(v_{n_j}^m)$, which converges a.e. in $[-A_m,A_m] \times [0,T]$ and weakly in $W^{2,1}_{p}([-A_m,A_m]\times [0,T])$, $1<p<2$, as $j \to +\infty$. By a standard diagonal extraction procedure, one can eventually extract a subsequence, say $(v_{n_k})$, such that $v_{n_k}$ and $\frac{\partial}{\partial \xi}v_{n_k}$ converge respectively to $v$ and $\frac{\partial}{\partial \xi}v$ almost everywhere in $\mathbb R \times [0,T]$ as $k \to +\infty$. After a new extraction, $v_{n_k} \to v$ in $W^{1,0}_\infty(\mathbb R\times [0,T])$ weak-$*$ and $W^{2,1}_\infty((\mathbb R\times [0,T] )\setminus \overline{Q}_{\rho}) $\,weak-$*$.  

It is not difficult to see that $v$ satisfies the properties of Lemmas \ref{l1a}-\ref{l3b}. Set $f_{\e}= \frac{\pr^2 v_\e}{\pr \xi^2}+\frac{\pr v_\e}{\pr \xi}$, $f_{\e} \leqslant 0 \in \mathbb{R}\times[0,T]$ (see Lemma \ref{l3b}). According to the above results, $f_{n''} \to f= \frac{\pr^2 v}{\pr \xi^2}+\frac{\pr v}{\pr \xi}$ in $L^\infty((\mathbb{R}\times [0,T] )\setminus \overline{Q}_{\rho}) $\,weak-$*$ which is non-negative in the distribution sense, and, hence, \eqref{inequality} holds. 

Since the sequence $\eta_{n''}$ is bounded in $C^1([0,T])$ (see Lemma \ref{lem_Lip}), a subsequence converges to some $\tilde  \eta$ in $C^0([0,T])$. More specifically, in the proof of Lemma \ref{lem_Lip}, we showed that $\frac{\pr v_\e}{\pr \xi}(\eta_\e(t),t)\leqslant -c$, where the constant $c$ is independent of $\e$. With the estimate of $\frac{\pr^2 v_\e}{\pr \xi^2}$ in Lemma \ref{l4}, we deduce that, for any $t>0$, there exists a small constant $\Upsilon$ independent of $\e$ such that, for $x<\Upsilon$, 
$$
v_{n''}(\eta_{n''}(t)+x,t)-v_{n''}(\eta_{n''}(t),t)\leqslant -\frac{c}{2}x,
$$  
and 
$$
v_{n''}(\eta_{n''}(t)-x,t)-v_{n''}(\eta_{n''}(t),t)\geqslant \frac{c}{2}x.
$$ 
Taking the limit as $n'' \to \infty$ and combining with  $v(\xi,\cdot)$ non-increasing in $\xi$, we see that $v<\gamma$ if $\xi >\tilde \eta(t)$ and $v>\gamma$ if $\xi <\tilde \eta(t)$.  This yields that $\tilde \eta=\hat \eta$ (see Corollary \ref{hat_eta}). 

Moreover, the convergence of $\eta_{n''}$ to $\hat \eta$ implies the almost everywhere convergence of $\sigma_{n''}(v_{n''})$ to $\sigma(v)$.
Hence, we have that $\mathcal L_{n''}[v_{n''}]$ converges  to $\mathcal L[v]$ in $L^\infty((\mathbb{R}\times [0,T] )\setminus \overline{Q}_{\rho}) $\,weak-$*$. This implies that $\mathcal L[v] \geqslant 0$. It is also easy to verify that $\mathcal L[v]=0$ whenever $v<1$. Thus, $v$ is a solution to \eqref{problem_evolution1}. 

Finally, let us check the growth condition as $\xi \to \pm \infty$: according to Lemmas \ref{l1a} and  \ref{lem_lower_bound},  $v_\e \leqslant \min(1,e^{-\xi})$ and  $v_\e \geqslant \tilde K-(\e\vee \e_\beta) e^{\delta t}$, respectively. At the limit $\e \to 0$, it holds almost everywhere $v(\xi,t) \leqslant \min(1,e^{-\xi})$ and $v(\xi,t)
\geqslant \tilde K$, $-\infty <\xi< +\infty, 0\leqslant t\leqslant T$. Therefore, $v(\xi,t)= 1$ when $\xi \leqslant  \tilde \kappa^*$ and  $v(\xi,t) \leqslant e^{-\xi}$ when $\xi \geqslant 0$. 
\end{proof}

Then, the uniqueness of $v$ given by Theorem \ref{thm_exist} is a direct consequence of the following theorem:

\begin{theorem}\label{uniqueness}
Let  $v_i\in\Big\{\bigcap_{\rho>0}W^{2,1}_\infty((\mathbb R\times [0,T] )\setminus \bar Q_\rho )\Big\}\cap
W^{1,0}_\infty(\mathbb R\times [0,T] )$ be a solution to \eqref{problem_evolution1} satisfying 
$$
\frac{\pr^2 v_i}{\pr \xi^2}+\frac{\pr v_i}{\pr \xi}\leqslant 0,
$$
for $i=1,2$.  Suppose that there exists $B_i>0$ such that $v_i=1$ for $\xi<-B_i$ and $v_i \leqslant e^{-\xi}$ for $\xi>B_i$, $i=1,2$. Then, it holds $v_1=v_2$.
\end{theorem}
\begin{proof}
	Denote by $F=\frac{1}{2} \left(\sigma^2(v_1)-\sigma^2(v_2)\right)\left(\frac{\pr^2 v_1}{\pr \xi^2}+\frac{\pr v_1}{\pr \xi}  \right)$ and $\mathcal L_2=-\frac{\partial }{\partial t} +\frac{1}{2} \sigma^2(v_2)\Big(\frac{\partial^2 }{\partial \xi^2}+\frac{\partial }{\partial \xi}\Big)+\delta\Big(\frac{\partial }{\partial \xi}+ 1\Big)$. We rewrite that 
	$$
	\min\{ \mathcal L_2[v_1]+F,1-v_1\}=0.
	$$
	Let $w=e^{-2\delta t}\left(v_1-v_2\right)$. We prove that $w \geqslant 0$. Due to the growth condition on $v_1$ and $v_2$,  it holds that 
	$\lim_{\xi \rightarrow \pm \infty} w(\xi,t)=0$ for $t \in [0,T]$. Therefore if this conclusion is not true,   $w$ will achieve a negative minimum at some point $(\xi^*,t^*)$. By the parabolic version of Bony's maximum principle, it holds that 
$$
\limsup_{(\xi,t) \rightarrow (\xi^*,t^*)} \mathop{ess}\left\{\frac{\pr w}{\pr t} -\frac{1}{2}\sigma^2(v_2) \frac{\pr^2 w}{\pr \xi^2}-\left( \frac{1}{2} \sigma^2(v_2)+\delta \right)\frac{\pr w}{\pr \xi} \right\} \leqslant 0
$$
This is equivalent to 
$$
\limsup_{(\xi,t) \rightarrow (\xi^*,t^*)} \mathop{ess}\left\{\mathcal L_2[v_1-v_2] \right\} \geqslant -\delta(v_1-v_2)>0.
$$
By the continuity of $v_i$, we derive $\sigma(v_1) \leqslant \sigma(v_2)$ in a small parabolic neighborhood of $(\xi^*,t^*)$. It follows that 
$$
\limsup_{(\xi,t) \rightarrow (\xi^*,t^*)} F(\xi,t) \geqslant 0.
$$
In this neighborhood, we also have that 
$$
\mathcal L_2[v_1]+F=0 \text{ and } \mathcal L_2[v_2]\geqslant 0,\text{a.e..}
$$
Therefore, 
$$
\limsup_{(\xi,t) \rightarrow (\xi^*,t^*)} \mathop{ess}\left\{\mathcal L_2[v_1-v_2] \right\} \leqslant \limsup_{(\xi,t) \rightarrow (\xi^*,t^*)} -F(\xi,t) \leqslant 0,
$$
which is a contradiction. Thus, we proved that $w \geqslant 0$. Similarly, the reverse inequality holds, which yields the uniqueness result.
\end{proof}

\subsection{Properties of the free boundaries}
For the original problem \eqref{problem_evolution1}, we already introduced formally the default boundary $\hat \kappa$ and the transit boundary $\hat \eta$, see System \eqref{FBP}. The goal of this subsection is to to define the free boundaries rigorously and prove some basic properties. 

\subsubsection*{The default boundary}
Let us remind that $v_\e \geqslant \tilde K-(\e\vee\e_\beta)e^{\delta t}$, see Lemma \ref{lem_lower_bound}. Taking the limit as $\e \to 0 $, this implies that $v \geqslant \tilde K$. Since $\tilde K=1$ for $\xi \leqslant \tilde  \kappa^*$, it holds that the set $\{ \xi\,|\,v(\xi,t)<1\}$ is bounded from below. Now, we are in position to define
\begin{eqnarray}\label{difkappa}
	\hat \kappa(t):=\inf \{ \xi\,|\, v(\xi,t)<1\}.
\end{eqnarray}
Then, $v\leqslant e^{-\xi}$  indicates that  $\hat \kappa(t)$  is also bounded from above. Thus, we will have the following result.
\begin{theorem}\label{def_kappa}
	For each $t\in(0,T]$,  $\hat \kappa(t)$ is well-defined, i.e. we have $-\infty<\hat \kappa(t)<\infty$. Moreover, $v(\xi,t)=1$ for $\xi \leqslant \hat \kappa(t)$ and $v(\xi,t)<1$ whenever $\xi>\hat \kappa(t)$.
\end{theorem}

\subsubsection*{The transit boundary} We remind that $\hat \eta \in C^0([0,T])$ is the limit of $\eta_\e$ (see Theorem \ref{thm_exist}). Thus, we will have the following theorem.
\begin{theorem} The initial positions of the free boundaries are as follows:
	$$\hat \eta(0)= \log \frac{1}{\gamma}, \quad \hat\kappa(0)=0.$$
	Furthermore, $\hat \kappa(t)$ and $\hat \eta(t)$ are non-increasing with respect to $t$.
\end{theorem}
\begin{proof}
On the one hand, we know that $\eta_{\e}(0)= -\log \gamma$ 
and $\eta_{\e}(t)$ decreasing, see Section \ref{approx_transit}. On the other hand, the properties of $\hat \kappa(t)$ follow from Theorem \ref{def_kappa} and the initial value of $v$.
\end{proof}

 In the following, we will prove the smoothness of the free boundaries.  Note that the uniform lower bound in Lemma \ref{l2} implies that there exists a constant $c$ such that $v_\e(\xi,t)\leqslant \frac{1+\gamma}{2}$ whenever  $\eta_\e(t)-\xi \leqslant c$.
 	Then, one can choose a smooth function $\zeta$ such that $\zeta(t)<\eta_\e(t)$ and  $\|\zeta-\eta_\e\|_{L^{\infty}[0,T]} \in [{c}/{4},{c}/{2}]$ for sufficiently small $\e$. Therefore, $\zeta$ separates the default boundary $\hat \kappa$ and the transit boundary $\hat \eta$. So, we can discuss them one by one with cut-off functions being applied when necessary.

We first study the default boundary.  The proof  is essentially the same as that in \cite{chengrong2006free}, where the authors proved the smoothness of free boundary in American option problem. Thus, we just give a sketch of the proof for readers' convenience. We make the change of variable $\xi=\zeta(t)+x$ and set $u(x,t)=v(\zeta(t)+x,t)$. For suitable $a,b\in \mathbb R$, we have that  $\zeta(t)+a \leqslant \hat \kappa(t) \leqslant \zeta(t)+b < \hat \eta(t)$. 
It holds that
$$
\frac{\pr u}{\pr t} \in L^{\infty}(t_1,t_2;H^1(a,b)),\frac{\pr^2 u}{\pr t^2} \in L^{2}(t_1,t_2;L^2(a,b)),
$$
which implies the continuity of $v_t$ at $\xi=\hat \kappa(t)$. From the definition of $\hat \kappa$, one can prove that  $\hat \kappa $ is continuous in $(0,T]$. 
Applying a result from Cannon et al. \cite{cannon1974continuous}, we will have that $\hat \kappa \in C^1((0,T])$. Then, we may use the theory of parabolic equations to improve the regularity of $v(\xi,\tau)$ by bootstrapping. Repeating the procedure yields the following result.
\begin{theorem}\label{thm_default}
	$\hat \kappa \in C^{\infty}((0,T])$.
\end{theorem}

Next, we consider the smoothness of the transit free boundary $\hat\eta(t)$. For this purpose, we need the following lemma of the parabolic diffraction problem. The proof is essentially similar to that in \cite{L}; hence, we just give a sketch.

\begin{lemma}\label{prop_smooth}
In the domain $Q=\{a<x<b,\ 0<t<T\}$, where $a<b$ are some constants, consider the following initial boundary problem
\begin{equation}\label{v0}
 \left \{
 \begin{split}
 	& u_t-(K_f(u_{x}+u))_x+f_1(x,t)u_x+f_2(x,t)u=0,\\
 	& u(a,t)=g_a(t),\ u(b,t)=g_b(t),\quad  u(x,0)=\phi(x),\\
& g_a(0)=\phi(a),\quad g_b(0)=\phi(b),
 \end{split}
 \right.
 \end{equation}
where $g_a,g_b\in C^{2}[0,T]$, $K_f(\phi_x+\phi)(x)\in C^{1}[a,b],$  $f_i(x,t)\in C([a,b]\times [0,T]),$  $i=1,2,$  $K_f=\left\{\begin{array}{ll}\mu_1, &\mbox{ if } x>f(t),\\ \mu_2, &\mbox{ if }  x<f(t),\end{array}\right.$, $a<f(t)<b$, for $t\in [0,T]$, $f(t)\in C^{0,1}(0,T)$,$a<f(t)<b$, for $t\in[0,T]$, and $\mu_1,\mu_2$ are positive constants. Then, the problem (\ref{v0}) admits a solution, and
$$
u(f(t)-,t)=u(f(t)+,t),\quad \mu_2(u+u_x)(f(t)-,t)=\mu_1(u+u_x)(f(t)+,t).
$$
Moreover, there exists a positive constant $C$ and $0<\alpha<1$ depend only on the given data such that
$$\|K_f(u+u_x)\|_{C^\alpha(Q)}\leqslant C.$$
\end{lemma}
\begin{proof}
Make the transformation $y=x-f(t)$, $v=ue^{y}$, then problem (\ref{v0}) satisfies
\begin{equation}\label{v00}
 \left \{
 \begin{split}
 	& v_t-(K_0(v_y))_y-f'(t)v_y+f_1v_y+(f_2-f_1)v=0,\\
 	& v(a-f,t)=g_a(t)e^{a-f},\ v(b-f,t)=g_b(t)e^{b-f},\quad  v(y,0)=\phi(y+f)e^{y}.
 \end{split}
 \right.
 \end{equation}
where $K_0=\mu_1$ if $y>0$,\; $\mu_2$ if $y<0$.
By well-known estimates for linear parabolic PDEs  with discontinuous coefficients whose principal part is in divergence form (see \cite[Chapter III, 5]{La}), and the  proof of \cite[Theorem 1.1]{L}, the claim of this lemma follows. 
\end{proof}
Now, we are in position to prove the smoothness of $\hat\eta$.
\begin{theorem}\label{thm_trans}
$\hat\eta \in C^{\infty}((0,T])$.
\end{theorem}
\begin{proof}
In a neighborhood of $\hat\eta$, $v$ satisfies the system
 \begin{equation}\label{v1}
 \left \{
 \begin{split}
 	&-\frac{\pr v}{\pr t}
+\frac12 \s^2_{H}\Big(\frac{\pr^2v}{\pr \xi^2} + \frac{\pr v}{\pr \xi}\Big)+\delta\Big(\frac{\pr v}{\pr \xi}
+v\Big)=0,  \quad \xi>\hat\eta(t),\\
 	&-\frac{\pr v}{\pr t}
+\frac12 \s^2_{L}\Big(\frac{\pr^2v}{\pr \xi^2} + \frac{\pr v}{\pr \xi}\Big)+\delta\Big(\frac{\pr v}{\pr \xi}
+v\Big)=0,   \quad \xi<\hat \eta(t),\\
 	&v(\hat\eta(t)+,t)=v(\hat\eta(t)-,t)=\gamma,\quad v_\xi(\hat\eta(t)+,t)=v_\xi(\hat\eta(t)-,t).
 \end{split}
 \right.
 \end{equation}
Thus, it holds that 

\begin{equation}\label{f1}
\hat\eta'(t)=-\frac{v_t(\hat\eta(t)+,t)}{v_\xi(\hat\eta(t)+,t)}=-\frac{v_t(\hat\eta(t)-,t)}{v_\xi(\hat\eta(t)-,t)},
\end{equation}
which means that
\begin{equation}\label{f2}
v_t(\hat\eta(t)+,t)=v_t(\hat\eta(t)-,t).
\end{equation}
Set $w=v_\xi$. From (\ref{v1}) and \eqref{f2}, it turns out that $w$ verifies the system
\begin{equation}\label{v2-2}
 \left \{
 \begin{split}
 	&-\frac{\pr w}{\pr t}
+\frac12 \s^2_{H}\Big(\frac{\pr^2w}{\pr \xi^2} + \frac{\pr w}{\pr \xi}\Big)+\delta\Big(\frac{\pr w}{\pr \xi}
+w\Big)=0,  \quad \xi>\hat\eta(t),\\
 	&-\frac{\pr w}{\pr t}
+\frac12 \s^2_{L}\Big(\frac{\pr^2w}{\pr \xi^2} + \frac{\pr w}{\pr \xi}\Big)+\delta\Big(\frac{\pr w}{\pr \xi}
+w\Big)=0,   \quad \xi<\hat \eta(t),\\
 	&w(\hat\eta(t)+,t)=w(\hat\eta(t)-,t), \quad \sigma^2_L(w_x+w)(\hat\eta(t)+,t)=\sigma^2_H(w_x+w)(\hat\eta(t)-,t).
 \end{split}
 \right.
 \end{equation}
According to the free boundary condition, $w$ satisfies a typical Verigin problem, see \cite{L,Y}. In particular, the $C^\infty$ regularity of the free boundary was proved in  \cite{L}. Therefore, we may obtain the same result for our problem in a similar manner. To see this, note that  the free boundary $\hat\eta$  is Lipschitz continuous and  satisfies
\begin{eqnarray}\nonumber
\hat\eta'(t)&=&-\frac{\frac12 \s^2_{H}\Big(w_\xi(\hat\eta(t)+,t)+ w(\hat\eta(t)+,t)\Big)+\delta\Big(w(\hat\eta(t)+,t)+\gamma
\Big)}{w(\hat\eta(t)+,t)}\\
&=&-\frac{\frac12 \s^2_{L}\Big(w_\xi(\hat\eta(t)-,t)+ w(\hat\eta(t)-,t)\Big)+\delta\Big(w(\hat\eta(t)-,t)+\gamma
\Big)}{w(\hat\eta(t)-,t)}, \label{etaup}
\end{eqnarray}
which is a kind of Stefan condition, see \cite{J1,J2, F1} for references. Applying Lemma \ref{prop_smooth} to problem (\ref{v2-2}) (up to some simple transformation), $w_\xi+w\in C^{\alpha}$ up to the free boundary. Furthermore, by Lemma \ref{l2a}, $w$ has a negative upperbound. Then, the right hand side of (\ref{etaup}) belongs to $C^{\alpha}$. This implies in turn $\hat\eta\in C^{1+\alpha}$. In this way, by an iteration process, one can further improve the regularity of $\hat \eta$ and shows eventually that it belongs to $C^\infty.$
\end{proof}

\section{Asymptotic Convergence}
In this section, we will prove that $v$ converges to the traveling wave solution as $t$ goes to $+\infty$. Since $\frac{\pr v}{\pr t}$ is non-positive, we see that, for any $t$, 
$$
0 \geqslant \int_0^t \frac{\pr v}{\pr t}(\xi,s)ds=v(\xi,t)-v(\xi,0) \geqslant \tilde K(\xi)-v(\xi,0).
$$
Note that for $\xi <\tilde \kappa^*$, $v(\xi,0)=\tilde K(\xi)=1$ and $\tilde K(\xi),v(\xi,0)\leqslant e^{-\xi}$ which implies the integrability of $\tilde K-v(\cdot,0)$ over $\mathbb R$. Thus, we have that 
$$
0\geqslant \int_{-\infty}^{\infty}\int_0^t \frac{\pr v}{\pr t}(\xi,s)dsd\xi \geqslant \int_{-\infty}^{\infty}( \tilde K(\xi)-v(\xi,0)) d\xi.
$$
Letting $t$ tend to infinity, we get that there exists a constant $C>0$ such that 
\begin{equation}\label{inequ_energy_t}
0\geqslant \int_{-\infty}^{\infty} \int_0^\infty\frac{\pr v}{\pr t}(\xi,s)dsd\xi \geqslant -C.
\end{equation}

Now let $v^n(\xi,t):=v(\xi,t+n)$ and consider $v^n$ as a sequence of functions defined on $\mathbb R \times[0,1]$. Lemmas \ref{l1a}-\ref{l4} indicate that it is a bounded sequence in $W^{2,1}_\infty(\mathbb R \times [0,1])$. As in the proof of Theorem \ref{thm_exist}, via a standard diagonal extraction procedure there exists a function $\bar K$ and a subsequence $n_j$ such that  such that $v_{n_j}$ and $\frac{\partial}{\partial \xi}v_{n_j}$ converge respectively to $\bar K$ and $\frac{\partial}{\partial \xi}\bar K$ almost everywhere in $\mathbb{R} \times [0,1]$. After a new extraction if necessary, 
$$
\frac{\pr v^{n_j}}{\pr t} \to \frac{\pr \bar K}{\pr t}, \quad \frac{\pr^2 v^{n_j}}{\pr \xi^2} \to \frac{\pr^2 \bar K}{\pr \xi^2} \;\text{ in } L^{\infty}(\mathbb R \times [0,1]) \;\mbox{weak-$*$},
$$
Since non-positivity is preserved under weak-$*$ convergence and  $\frac{\pr v}{\pr t} \leqslant 0$, one can deduce that $\frac{\pr \bar K}{\pr t} \leqslant 0$. Since \eqref{inequ_energy_t} implies that $\int_0^1 \int_{-\infty}^{\infty} v^{n_j}(\xi,t)d\xi dt= \int_{n_j}^{n_j+1} \int_{-\infty}^{\infty} v(\xi,t)d\xi dt \rightarrow 0$ as $n_j \rightarrow 0$, we have that $\int_0^1 \int_{-\infty}^{\infty} \frac{\pr \bar K}{\pr t} d \xi dt=0$. Combining with the non-positivity of $\frac{\pr \bar K}{\pr t}$, it follows that $\frac{\pr \bar K}{\pr t} \equiv 0$ which means that $\bar K$ is only a function of $\xi$. Then, the following properties pass from $v$ to $\bar K$,
  $$
 \tilde K \leqslant \bar K \leqslant \min\{1,e^{-\xi} \},\;\, \frac{d \bar K}{d \xi} \leqslant 0, \, \text{ and }\,  \frac{d^2 \bar K}{d \xi^2}+ \frac{d \bar K}{d \xi} \leqslant 0.
  $$
  
  Since $\hat \eta(\cdot)$ and $\hat \kappa(\cdot)$ are also non-increasing with respect to $t$,  they also admit limits at $\infty$, which are denoted as $\bar \eta$ and $\bar  \kappa$ respectively. Then, one can verify that $\bar K(\bar \eta)=\gamma$ and $\bar K(\bar \kappa)=1$. For any interval $I$ such that $\bar I \subset (\bar \kappa,\bar \eta)$, there exists $T$ such that $\bar I \subset (\hat \kappa(t),\hat \eta(t))$ for any $t>T$. In $I$, it holds that 
  $$
  -\frac{\pr v^n}{\pr t}+\frac{1}{2}\sigma_L^2\left(\frac{\pr^2 v^n}{\pr \xi^2}+\frac{\pr v^n}{\pr \xi}  \right)+\delta\left(\frac{\pr v^n}{\pr \xi}+v^n  \right)=0.
  $$
  Taking subsequence $n_j$, we derive that 
  $$
  \frac{1}{2}\sigma_L^2\left(\frac{d^2 \bar K}{d \xi^2}+\frac{d \bar K}{d \xi}  \right)+\delta\left(\frac{d \bar K}{d \xi}+\bar K  \right)=0, \text{ for } \xi \in I.
  $$
  Since $I$ is arbitrary, it holds that 
  $$
  \frac{1}{2}\sigma_L^2\left(\frac{d^2 \bar K}{d \xi^2}+\frac{d \bar K}{d \xi}  \right)+\delta\left(\frac{d \bar K}{d \xi}+\bar K  \right)=0, \text{ for } \bar \kappa<\xi<\bar \eta .
  $$
  Similarly, we can also show that 
  $$
   \frac{1}{2}\sigma_H^2\left(\frac{d^2 \bar K}{d \xi^2}+\frac{d \bar K}{d \xi}  \right)+\delta\left(\frac{d \bar K}{d \xi}+\bar K  \right)=0, \text{ for } \xi>\bar \eta .
  $$
  Note that $\tilde K \leqslant \bar K \leqslant \min\{1,e^{-\xi} \}$ implies that $\lim_{\xi \rightarrow \infty} e^{\xi}\bar K(\xi)=1$. Combining with the fact that $\bar K \in C^{1+\alpha}$, we see that it is a solution to \eqref{steady1}, i.e.
  \begin{equation*}
  	\left \{
  	\begin{split}
  		&\frac{d^2 \bar K}{d \xi^2}+ \frac{d\bar K}{d \xi}+c_H(\frac{d\bar K}{d \xi} +K)=0,  \xi>\bar \eta,\\
  		& \frac{d^2 \bar K}{d \xi^2}+\frac{d \bar K}{d \xi}+c_L(\frac{d\bar K}{d \xi} +K)=0,  \bar \kappa <\xi<\bar \eta,\\
  		&\bar K(\bar \kappa)=1,\frac{d\bar K}{d \xi}(\bar \kappa)=0,\\
  		&\bar K(\bar \eta)=\bar K(\bar \eta^*-)=\gamma, \frac{d\bar K}{d \xi}(\bar \eta+)=\frac{d\bar K}{d \xi}(\bar \eta-),\\
  		&\lim_{\xi \rightarrow \infty} e^{\xi}\bar K(\xi)=1.
  	\end{split}
  	\right.
  \end{equation*}
Then, interior estimate implies that $\bar K$ is smooth in $(\bar \kappa,\bar \eta)$ and $(\bar \eta,\infty)$. Now, from the uniqueness of the solution, we derive that $\bar K=K$. Since any sub-sequential limit must be same, the full sequence must converge as $n$ goes to $\infty$. We have proved the local convergence of $v$. But, noting that $v(\xi,t)\equiv1$ for $\xi<\tilde \kappa^*$ and $v(\xi,t )\leqslant e^{-\xi}$, the convergence is also uniform over $\mathbb R$. Finally, we prove the following result.
\begin{theorem}
As $t$ goes to $+\infty$, $v(\cdot,t)$ converges uniformly to $K$. 
\end{theorem}
\section{Numerical Results}
In this section, we will  give some numerical results for illustration. As $u$ represents the value of the bond, we will come back to \eqref{problem_evolution} instead of \eqref{problem_evolution1} which will give us more clear financial meaning.
\subsection{Numerical Scheme}
As our problem is non-standard, we will introduce the numerical scheme first. To solve the free boundary problem, we use an explicit-implicit finite difference scheme combined with Newton iteration to solve the penalized equation. The first step is to discretize the equation. Let $t_i=i\Delta t,i=0,1,...,M$, and $\xi_j=j\Delta \xi,j=0,\rm 1,\pm2,...,\pm N$. $U_{i,j}$ will be the approximation of the solution $u$ of \eqref{problem_evolution}  at mesh point $(t_i,\xi_j)$. Consider the approximating penalized equation 
\begin{equation*}
\left\{
\begin{split}
&-\frac{\pr u}{\pr t}+\frac{1}{2}\sigma^2_\e(u,\xi)(\frac{\pr^2 u}{\pr \xi^2}-\frac{\pr u}{\pr \xi})+\delta \frac{\pr u}{\pr \xi}= \e^{-1}(u-e^{\xi})^+,\quad\xi \in [-N\Delta \xi,N\Delta \xi ],t\geqslant 0;\\
&u(\xi,0)=\min\{1,e^{\xi}\};\\
&u(N\Delta \xi,t)=1,u(-N\Delta \xi,t)=0.
\end{split}
\right.
\end{equation*}
 Here $\sigma_\e(u,\xi)=\sigma_H+(\sigma_L-\s_H)H_\e(u-\gamma e^{\xi})$ with $H_\e$ be a proper smooth function. For numerical convenience, we use the penalty function $\e^{-1}(u-e^\xi)^+$. In the numerical experiment, we choose 
 $$
 H_\e(z)=\left\{ 
 \begin{split}
 &0, z\leqslant -\e;\\
 &6\e^{-5}z^5+15 e^{-4}z^4+10 \e^{-3}z^3+1, -\e<z<0;\\
 &1,z\geqslant 0,
 \end{split}
 \right.
 $$
 as proposed in \cite{li2018convergence}. Note that the left hand side is a nonlinear operator since  coefficients depend on $u$. In the  numerical implement, we determine these coefficients with function value from previous time step. For illustration, let us perform discretization  at $(t_i,\xi_j)$. Denote by $\sigma_{i,j}:=\sigma_\e(U_{i,j},\xi_j)$. The first order term is discretized by the upwind scheme, i.e.
 $$
 (\delta-\sigma_{i-1,j}) \frac{\pr u}{\pr \xi}(t_i,\xi_j)\approx\left\{
\begin{split}
&(\delta-\sigma_{i-1,j}) \frac{U_{i,j+1},U_{i,j}}{\Delta \xi},\text{ if $\delta-\sigma_{i-1,j} \geqslant 0 $;}\\
&(\delta-\sigma_{i-1,j}) \frac{U_{i,j},U_{i,j-1}}{\Delta \xi},\text{ if $\delta-\sigma_{i-1,j} < 0 $.}
\end{split} 
 \right.
 $$
 We use the fully implicit approximation to the temporal term 
 $$
 \frac{\pr u}{\pr t}(t_i,\xi_j) \approx \frac{U_{i,j}-U_{i-1,j}}{\Delta t},
 $$ 
 and 
 the usual discretization for the second order term 
 $$
 \frac{\pr^2 u}{\pr \xi^2} \approx \frac{U_{i,j+1}+U_{i,j-1}-2U_{i,j}}{(\Delta \xi)^2}.
 $$
 Thus, given function value $U_{i-1,\cdot}$ at previous time step, current value $U_{i,\cdot}$ is obtained by solving the following equation 
 \begin{equation}\label{FD_nonlinear}
 [A_i U_{i,\cdot}]_j=\e^{-1}(U_{i,j}-e^{\xi_j})^+
 \end{equation}
 for $j=0,\pm 1,\pm 2,...,\pm N$. Here the matrix $A_i$ is determined by $U_{i-1,\cdot}$ and is a sparse $M$-matrix due to our discretization scheme. 
 
 Now, we have to solve the nonlinear equation \eqref{FD_nonlinear}.  We adopt the method used by \cite{forsyth2002quadratic} to value American options. For illustration, let us recall the classical Newton iteration for finding the root of a convex function $f$. Given an initial guess, the point is updated as 
 $$
 z_{n+1}=z_n-\frac{f(z_n)}{f'(z_n)}
 $$
 which is equivalent to say that $z_{n+1}$ solves
 $$
 f(z_n)+f'(z_n)(z-z_n)=0.
 $$
 It is easy to see that the left hand side of above equation is an first order approximation of $f$ at $z_n$. Similarly, we can solve \eqref{FD_nonlinear} with Newton iteration. Denote $U^k_{i,j}$ as the approximation at $(t_i,\xi_j)$ for $k$th iteration. Then, $U^k_{i,\cdot}$ solves the linearized equation
 \begin{equation}\label{FD_linearized}
 	 [A_i U^k_{i,\cdot}]_j=\e^{-1}(U^{k-1}_{i,j}-e^{\xi_j})^++\e^{-1}1_{\{ U^{k-1}_{i,j}-e^{\xi_j}>0 \}}(U^{k}_{i,j}-U^{k-1}_{i,j}).
 \end{equation}
When the difference between $U^k_{i,\cdot}$ and $U^{k-1}_{i,\cdot}$ is small enough, we stop the iteration and set $U_{i,\cdot}$ equals $U^k_{i,\cdot}$. Moreover, the initial guess $U^0_{i,\cdot}$ is chosen to be $U_{i-1,\cdot}$. 

In summary, we have the following iterative algorithm.
\begin{breakablealgorithm} \caption{Explicit-Implicit Finite-difference Iterative Algorithm}
	\begin{algorithmic} 	
		\Require{$N,M,L,\Delta t,\Delta \xi$, smooth function $H_\e(\cdot)$ and tolerance $tol$}
		\State{Initialize $U_{0,j}=\min\{1,e^{\xi_j}\}$}
		\For{ $i=1,2,...,M$}
		\State{Construct the matrix $A_{i}$ according to upwind scheme with 
			$$
			\sigma_{i,j}:=\sigma_\e(U_{i,j},\xi_j)
			$$
		}
		\State{Set $U^0_{i,\cdot}=U_{i-1,\cdot}$}
		\While{True}
		\State{Solve
			\begin{equation*}
				 [A_i U^k_{i,\cdot}]_j=\e^{-1}(U^{k-1}_{i,j}-e^{\xi_j})^++\e^{-1}1_{\{ U^{k-1}_{i,j}-e^{\xi_j}>0 \}}(U^{k}_{i,j}-U^{k-1}_{i,j}).
			\end{equation*}
		} 
		\State{If $\frac{\|U^{k}_{i,\cdot}-U^{k-1}_{i,\cdot}\|_{\infty}}{\max\{1,\|U^{k-1}_{i,\cdot}\|_{\infty}\}}<tol$, Quit}
		\EndWhile
		\State{Set $U_{i,\cdot}=U^k_{i,\cdot}$.}
	    \EndFor
	\end{algorithmic}
\end{breakablealgorithm}
\subsection{Numerical Results} 
In the numerical experiment, we set the model parameters as $\delta=0.03,\sigma_l=0.3,\sigma_h=0.2$ and $\gamma=0.6$. For discretization, we have $\Delta t=0.01,\Delta \xi=0.001$ and $N=10^3$. We also choose $\e=10^{-8}$, $tol=10^{-4}$.  Having numerically solved \eqref{eq_diff}, we are able to plot the traveling equation for \eqref{problem_evolution}, which is  $e^{\xi}K(\xi)$. 
\begin{figure}[H]
	\centering
	\includegraphics[width=0.60\textwidth]{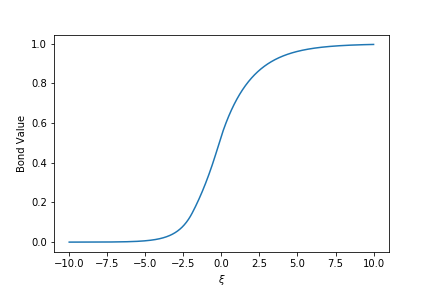}
	\caption{Typical traveling wave equation }
	\label{TW}
\end{figure}
Next, we plot the numerical solution for \eqref{problem_evolution} and compared it with the traveling wave equation in Figure \ref{Sols}.
\begin{figure}[H]
	\centering
	\includegraphics[width=0.60\textwidth]{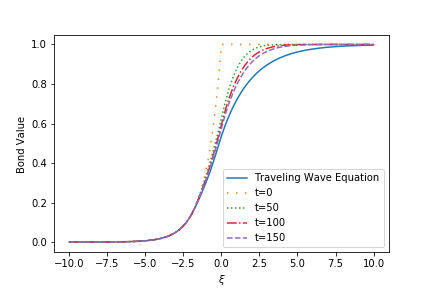}
	\caption{Solutions of the free boundary problem  at time $t=0,50,100,150$. }
	\label{Sols}
\end{figure}
 \noindent It seems  that the solution will converge to the traveling wave equation as $t$ goes to infinity as the theoretical result indicates. To numerically check this, we compute the solution for large time $t$ and plot the error between the solution and the traveling wave equation. The result is shown in Figure \ref{Error}. The error is defined as the supreme norm between the traveling wave equation $K$ and the value function at time $t$. We see that the error is monotone decreasing with respect to $t$. The final error is about $3.6\times10^{-3}$ at time $t=1500$. 
\begin{figure}[H]
	\centering
	\includegraphics[width=0.60\textwidth]{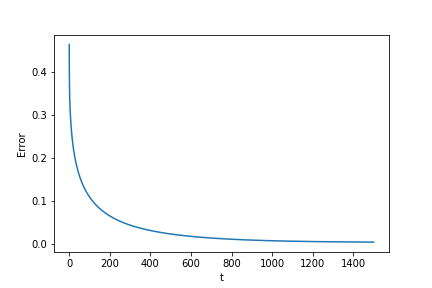}
	\caption{Differences between the free-boundary problem and traveling wave equation. }
	\label{Error}
\end{figure}
Finally, we plot the default and transit boundaries as a function of $t$ and compare them with those of traveling wave equation. The result is shown in Figure \ref{Boundaries}. It is clear that the boundaries are decreasing with respect to $t$ which is consistent with our previous theoretical analysis. We also see  the convergence of two boundaries.  
\Copy{r1major2}{
\begin{figure}[H]
	\centering
	\includegraphics[width=0.60\textwidth]{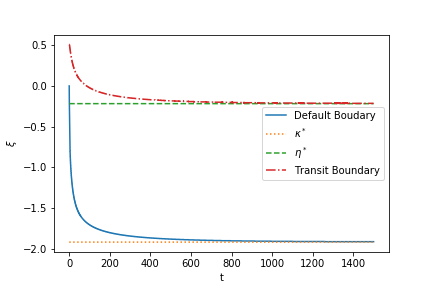}
	\caption{Default and transit boundaries as a function of $t$ }
	\label{Boundaries}
\end{figure}
}
\bibliographystyle{elsarticle-num}

\end{document}